\documentclass[11pt]{article}
\usepackage{mathenv,amsfonts, amsmath, amssymb,  graphics, epsfig,
mathrsfs, psfrag,color}
\graphicspath{{./dessins/}}
\newtheorem{theorem}{Theorem}[section]
\newtheorem{lemma}[theorem]{Lemma}
\newtheorem{prop}[theorem]{Proposition}

\newtheorem{rem}[theorem]{Remark}
\newtheorem{definition}[theorem]{Definition}
\newenvironment{proof}[1][]{{\medskip\noindent\it Proof #1: }}{
\hfill$\square$\\ \\}
\numberwithin{equation}{section}
\newcommand{\ranglemp}[1]{\rangle_{-\frac{1}{2},\frac{1}{2},#1}}
\newcommand{\ranglepp}[1]{\rangle_{\frac{1}{2},#1}}
\newcommand{\ranglemm}[1]{\rangle_{-\frac{1}{2},#1}}
\newcommand{\psie}[1]{\hat{\mathsf e}_{#1}}
\newcommand{\curve}{{\mathscr C}}
\newcommand{\SL}[1]{{\mathscr S}_{#1}}
\newcommand{\Hp}[1]{H^\frac{1}{2}(#1)}
\newcommand{\Hm}[1]{H^{-\frac{1}{2}}(#1)}

\renewcommand{\Re}{{\rm Re}\,}
\renewcommand{\Im}{{\rm Im}\,}
\def\presuper#1#2%
{\mathop{}%
\mathopen{\vphantom{#2}}^{#1}%
\kern-\scriptspace%
#2}

\begin{document}
\title{Conformal mapping  for cavity inverse problem: an explicit reconstruction formula}
\author{Alexandre
	Munnier\footnote{Universit\'e de Lorraine and CNRS, Institut \'Elie Cartan de
		Lorraine, UMR 7502, Vand\oe uvre-l\`es-Nancy, F-54506, France.
		\texttt{alexandre.munnier@univ-lorraine.fr}}\and Karim Ramdani\footnote{Inria,
		Villers-l\`es-Nancy, F-54600, France. \texttt{karim.ramdani@inria.fr}}}
\date{\today}
\maketitle
\begin{abstract}
	In this paper, we address a classical case of the Calder\'on (or conductivity)
	inverse problem in dimension two. We aim to  recover the location and the shape of a
	single cavity $\omega$ (with boundary $\gamma$)  contained in a domain
	$\Omega$ (with boundary $\Gamma$) from the knowledge of 
	the Dirichlet-to-Neumann (DtN) map $\Lambda_\gamma: f \longmapsto \partial_n
	u^f|_{\Gamma}$,  
	where $u^f$ is harmonic in $\Omega\setminus\overline{\omega}$, $u^f|_{\Gamma}=f$ and 
	$u^f|_{\gamma}=c^f$, $c^f$ being the constant such that $\int_{\gamma}\partial_n
	u^f\,{\rm d}s=0$. We obtain an explicit formula for the complex coefficients $a_m$ arising in the expression of the Riemann map
	$z\longmapsto a_1 z + a_0 + \sum_{m\leqslant -1} a_m z^{m}$
	that conformally maps the exterior of the unit disk onto the exterior of
	$\omega$. This formula is derived by using two ingredients: a new factorization result of the DtN map 	and the so-called generalized P\'olia-Szeg\"o tensors (GPST) of the cavity. As a byproduct of our analysis, we also prove the analytic dependence of the coefficients $a_m$ with respect to the DtN. Numerical results are provided to illustrate the efficiency and simplicity of the method.  
\end{abstract}
\section{Introduction}
%==============================================================================

Let $\Omega$ be a simply connected open bounded set in $\mathbb R^2$ with Lipschitz  boundary $\Gamma$. Let $\sigma$ be a positive function in $L^\infty(\Omega)$ and consider the elliptic boundary value problem:
\begin{subequations}
\label{calderon}
\begin{alignat}{3}
-\nabla\cdot(\sigma\nabla u)&=0 &\quad&\text{in }\Omega\\
u&=f&&\text{on }\Gamma.
\end{alignat}
Calder\'on's inverse conductivity problem \cite{Cal80} can be stated as follows: Knowing the Dirichlet-to-Neumann (DtN) map $\Lambda_\gamma:f\longmapsto \partial_nu^f$, is it possible to recover the conductivity $\sigma$?
\end{subequations}

\begin{figure}[h]%[htp]
\centerline{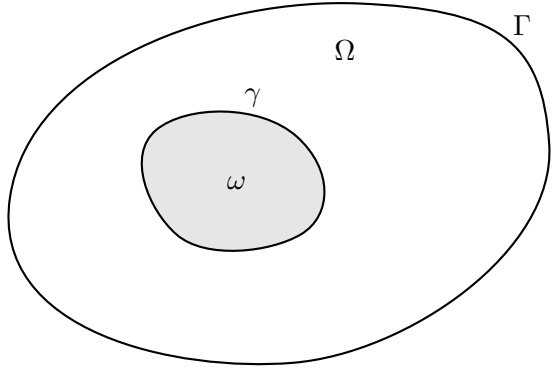}
\caption{The geometry.}
\label{fig:cavity}
\end{figure}

In this work, we focus on the particular case of piecewise conductivity with infinitely high contrast (see for instance Friedman and Vogelius \cite{FriVog89} who considered this problem in the case of small inclusions).  
%\begin{equation}
%\sigma = 1 + (\alpha-1)\chi_\omega,
%\end{equation}
%\end{subequations}
%where $\alpha\to +\infty$.
%... isolated inclusions, infinite conductivity...
More precisely, we suppose that $\Omega$ contains a cavity $\omega$, where $\omega$ is an open connected set with Lipschitz boundary $\gamma$ and such that $\overline{\omega}\subset\Omega$  (see Figure~\ref{fig:cavity}).
We denote by $n$ the unit normal to
$\Gamma\cup \gamma$ directed towards the exterior of $\Omega\setminus\overline{\omega}$.  

For every $f$ in $H^{\frac{1}{2}}(\Gamma)$, we denote by  $(u^f,c^f)\in
H^1(\Omega\setminus\overline{\omega})\times\mathbb R$ the solution to the Dirichlet
problem:
\begin{subequations}
\label{main_problem}
\begin{alignat}{3}
-\Delta u^f&=0&\quad&\text{in }\Omega\setminus\overline{\omega}\\
u^f&=f&&\text{on }\Gamma\\
u^f&=c^f&&\text{on }\gamma,
\end{alignat}
where $c^f$ is the unique constant such that:
\begin{equation}
\label{free_circ}
\int_\gamma \partial_n u^f\,{\rm d}\sigma =0.
%\langle \partial_nu^f,1\rangle_{H^{-\frac{1}{2}}(\Gamma)\times
%H^{\frac{1}{2}}(\Gamma)}= 0.
\end{equation}
\end{subequations}
%In dimension 3, the constraint \eqref{free_circ} is no longer required in our analysis and the constant $c^f$ is chosen to be equal to 0.
Problem \eqref{main_problem} is well-posed and its solution is the limit of the solution of \eqref{calderon} for piecewise constant conductivity, when the contrast between the cavity and the background tends to infinity (see Proposition~\ref{propppp} of the Appendix for a precise statement of this classical result and for the proof, which is given for the sake of completeness).

Loosely  speaking (the exact functional framework will be made precise later
on), the inverse problem considered throughout this paper is the following:
\emph{knowing the Dirichlet-to-Neumann (DtN) map $\Lambda_\gamma:f\longmapsto
\partial_nu^f$, how to reconstruct the cavity $\omega$?} 
	%-----------------------------------------------
\begin{rem}
In dimension 2, it is classical to see $u^f$ as the harmonic conjugate function of $v^f$, i.e. the solution to:
\begin{subequations}
\label{main_problem_potential}
\begin{alignat}{3}
-\Delta v^f&=0&\quad&\text{in }\Omega\setminus\overline{\omega}\\
\partial_n v^f&=\partial_\tau f&&\text{on }\Gamma\\
\partial_n v^f&=0&&\text{on }\gamma,
\end{alignat}
\end{subequations}
where $\tau:=n^\perp$ is the unit tangent vector to $\Gamma$. The function $u^f$ is usually referred to as the stream function 
associated to the potential function $v^f$. On $\Gamma$, $\partial_n u^f=-\partial_\tau v^f$ and therefore, the knowledge of $\Lambda_\gamma$ (i.e. the DtN for $u^f$)
is equivalent to the knowledge of the Neumann-to-Dirichlet map for $v^f$. 
\end{rem}

Classically for inverse problems, the questions of uniqueness, stability and reconstruction have been studied in the literature for cavities identification. Regarding uniqueness, it is well-known that one pair $(f,\partial_nu^f)$ of Cauchy data uniquely determines the geometry of the cavity for a Dirichlet boundary condition (see Kress \cite{Kre04}) or a Neumann boundary condition (see Alessandrini and Rondi \cite{AleRon01}). For Robin type condition,  Bacchelli \cite{Bac09} proved that two excitations  $f_1$ and $f_2$ uniquely determine the cavity provided they are linearly independent and one of them is positive. Concerning stability, as shown by Mandache \cite{Man01}, logarithmic stability is best possible (see also Alessandrini and Rondi  \cite{AleRon01} and references therein). Among the reconstruction methods available in the literature for shape identification, one can distinguish two classes of approaches: iterative and non iterative methods (see for instance the survey paper by Potthast \cite{Pot06} for an overview of  reconstruction methods). In the first class of methods, one computes a sequence of approximating shapes, generally by solving at each step the direct problem and using minimal data (typically only one or several pairs of Cauchy data, and not the full DtN map). Among these approaches, we can mention those based on optimization \cite{Bor02,CapFehGou09}, on the reciprocity gap principle \cite{KreRun05,IvaKre06,CakKre13}, on the quasi-reversibility \cite{BouDar10,BouDar14} or on conformal mapping \cite{AkdKre02,Kre04,HadKre05,HadKre06,HadKre10,Kre12,HadKre14}. 

The second class of methods covers non iterative methods which are generally based on the construction (from the measurements) of an indicator function of the inclusion(s). These sampling/probe methods do not need to solve the forward problem, but require the knowledge of the full DtN map. Among these reconstruction techniques, let us mention --with no claim as to completeness-- the enclosure and probe method of Ikehata \cite{Ike98,Ike00,Ike00b,IkeSil00,ErhPot06}, Kirsch's Factorization method \cite{BruHan00,HanBru03,Kir05} and Generalized Polya-Szeg\"o Tensors in the case of small inclusions \cite{AmmKan04,AmmKan06,AmmKan07,AmmGarKan14}.

Our purpose in this paper is to propose a new non iterative reconstruction method that combines some of the ingredients used in earlier works, namely: a new factorization result (Theorem~\ref{first_decom}), Generalized Polya-Szeg\"o Tensors and conformal mapping. The main feature of our reconstruction method is that we end up with an {\em explicit reconstruction formula} (Theorem~\ref{thm:expl}) for the complex coefficients $a_k$ arising in the expression of the Riemann map  $z\longmapsto a_1 z + a_0 + \sum_{m\leqslant -1} a_m z^{m}$ that conformally maps the exterior of the unit disk onto the exterior of $\omega$. Let us emphasize that these reconstruction formulae also yield the analytic dependence of the coefficients with respect to the DtN.

The proposed reconstruction algorithm can --in principle-- be adapted to other boundary conditions. However, such as most direct reconstruction methods, it requires the knowledge of the full DtN map and so far, it is limited to the two-dimensional case due to the use of conformal mapping. 

The paper is organized as follows: we present in Section~\ref{sect:EI} a boundary integral formulation of the problem. Section~\ref{sect:MAIN} is devoted to the derivation of the reconstruction formula, using a new factorization result and GPST. Some issues about stability are also discusses therein. Finally, some numerical results are given in Section~\ref{sect:NUM}.

%==============================================================================
\section{Boundary integral formulation}
\label{sect:EI}
%==============================================================================
\subsection{Background on single layer potential}
In this section, we collect some well known facts of potential theory, and more especially on single layer potential, that are crucial for our method. For more details and for the proofs, we refer the interested reader to the monographs of McLean \cite{McL00}, Steinbach \cite{Ste08b} or Hsiao and Wendland \cite{HsiWen08}.

Throughout the article, we shall denote by 
$$
G(x)= -\frac{1}{2\pi}\log|x|
$$
%$$G(x)= \begin{cases}\displaystyle\frac{1}{2\pi}\ln\left(\frac{1}{|x|}\right)&\text{if }N=2\\
%\displaystyle\frac{1}{4\pi}\frac{1}{|x|}&\text{if }N=3,
%\end{cases}$$
the fundamental solution of the operator $-\Delta$ in $\mathbb R^2$. 

Let $\mathsf C_i$ be a  bounded, Lipschitz domain (see \cite[Definition 3.28]{McL00}) and denote by $\mathscr C$ its boundary. Let $n$ be the unit normal to $\mathscr C$ directed towards the exterior of $\mathsf C_i$.  

The exterior of $\mathsf C_i$ is denoted $\mathsf C_e:=\mathbb R^2\setminus\overline{\mathsf C}_i$. Given a function $u$ in  $H^{1}_{\ell oc}(\mathbb R^2)$, we denote by $u_i$ and $u_e$ its  restrictions respectively to $\mathsf C_i$ and $\mathsf C_e$ and by $[u]_{\mathscr C}=u_i|_{\mathscr C}-u_e|_{\mathscr C}\in \Hp{\mathscr C}$ its jump across $\mathscr C$.  We also define similarly the jump of the normal derivative: $[\partial_{n}u]_{\mathscr C}=(\partial_{n}u_i)|_{\mathscr C}-(\partial_{n}u_e)|_{\mathscr C}\in \Hm{\mathscr C}$.

\begin{definition}
	For every $\hat q\in H^{-\frac{1}{2}}(\mathscr C)$, we denote by $\mathscr
	S_{\mathscr C} \hat q$ the single layer potential associated to the density $\hat q$. 
\end{definition}
The single layer potential $\mathscr S_{\mathscr C} \hat q$ defines a harmonic
function in $\mathbb R^2\setminus \mathscr C$. The operator $\mathscr
S_{\mathscr C}$ is an integral operator with weakly singular kernel, so that for $\hat q\in L^\infty(\mathscr C)$ for instance and  $x\in\mathbb R^2\setminus\mathscr C$, it reads:
$$\mathscr S_{\mathscr C} \hat q(x)=\int_{\mathscr C} G(x-y) \hat q(y)\,{\rm d}\sigma_y.$$
Moreover, the single layer potential defines a bounded linear operator from $H^{-\frac{1}{2}}(\mathscr C)$ into $H^1_{\rm \ell oc}(\mathbb R^2)$, and $\mathscr S_{\mathscr C} \hat q$ admits the following asymptotic behavior at infinity  (see for instance \cite[p.~261]{McL00})
\begin{equation}\label{asymptSC}
\mathscr S_{\mathscr C} \hat q(x) = -\frac{1}{2\pi}\langle \hat q, 1\ranglemp{{\mathscr C}}\, \log |x| + O(|x|^{-1}),
\end{equation}
where  $\langle \cdot, \cdot\ranglemp{{\mathscr C}}$ stands for the duality brackets between $\Hm{\mathscr C}$ and $\Hp{\mathscr C}$. This shows in particular that $\mathscr S_{\mathscr C} \hat q\in H^1(\mathbb R^2)$ if and only if $\hat q$ belongs to the function space 
$$
{\widehat H}({\mathscr C}):= \{\hat q\in H^{-\frac{1}{2}}({\mathscr C})\,:\,\langle \hat q,1\ranglemp{{\mathscr C}} =0\}.
$$
We also recall that the single layer potential satisfies the following classical jump conditions 
\begin{equation}\label{eq:jump}
[\mathscr S_{\mathscr C} \hat q\, ]_{\mathscr C}=0, \qquad \qquad 
\left[ \partial_{n}(\mathscr S_{\mathscr C} \hat q) \right]_{\mathscr C}=\hat q.
\end{equation}
%The function $\mathscr S_\mathscr C \hat q$ can also be seen as the unique solution in $H^1_{\rm \ell oc}(\mathbb R^2)$ of the (well-posed) transmission problem:
%\begin{subequations}
%	\label{def_trans}
%	\begin{alignat}{3}
%	-\Delta u&=0&\quad&\text{in }\mathbb R^2\setminus\mathscr C\\
%	[u]&=0&&\text{on }\mathscr C\\
%	[\partial_nu]&=\hat q&&\text{on }\mathscr C,\\
%	u&=\mathcal O(G(x))&&\text{as }|x|\to+\infty,
%	\end{alignat}
%\end{subequations}
Let us focus now on the trace of the single layer potential.
\begin{definition}
For every $\hat q\in H^{-\frac{1}{2}}(\mathscr C)$, we denote by ${\mathsf S}_{\mathscr C}\hat q$
the trace of the  single layer operator $\mathscr S_\curve \hat q$ on $\mathscr
C$.
\end{definition}
The operator ${\mathsf S}_{\mathscr C}$ is an integral operator with weakly singular
kernel as well. For $\hat q\in L^\infty(\mathscr C)$ and 
for every $x\in \mathscr C$, it reads:
$${\mathsf S}_{\mathscr C} \hat q(x)=\int_{\mathscr C} G(x-y) \hat q(y)\,{\rm d}\sigma_y.$$
The trace ${\mathsf S}_{\mathscr C}$ of the single layer operator defines a bounded linear operator from $H^{-\frac{1}{2}}(\mathscr C)$ into $H^{\frac{1}{2}}(\mathscr C)$. Furthermore, using Green's formula and the asymptotics \eqref{asymptSC}, we can easily prove the identity
\begin{equation}
\label{equiv:decroi}
\langle \hat q,\mathsf S_\Gamma \hat q  \ranglemp{{\mathscr C}} = \int_{\mathbb R^2} \left|\nabla (\mathscr S_{\mathscr C} \hat q)\right|^2Ê{\rm d}x<+\infty, \qquad \forall \hat q\in {\widehat H}({\mathscr C}).
\end{equation}
Following \cite[Theorem 8.15]{McL00}, we also introduce the following particular
density and constant which will play a crucial role in our analysis.
\begin{definition}
\label{def:psieq}
The equilibrium density for ${\mathscr C}$ is the unique density $\hat{\mathsf e}_{\mathscr C}\in H^{-\frac{1}{2}}({\mathscr C})$ such that $ {\mathsf S}_{\mathscr C}
\hat{\mathsf e}_{\mathscr C}$ is constant on ${\mathscr C}$ and 
$$\langle \hat{\mathsf e}_{\mathscr C},1\ranglemp{{\mathscr C}}=1.$$
The logarithmic capacity ${\rm Cap}({\mathscr C})$ of ${\mathscr C}$ is defined
as being the positive constant:
$$
{\rm Cap}({\mathscr C})=
\exp\left(-2\pi {\mathsf S}_{\mathscr C} \hat{\mathsf e}_{\mathscr C}\right).
$$
%	$${\rm Cap}({\mathscr C})=\begin{cases}
%\exp\left(-2\pi {\mathsf S}_{\mathscr C} {\mathsf e}_{\mathscr
%	C}\right)&\text{when }N=2,\\
%	1/{\mathsf S}_{\mathscr C} {\mathsf e}_{\mathscr
%	C}&\text{when }N=3.
%	\end{cases}$$
%	(When $N=3$, it can be proved that ${\mathsf S}_{\mathscr C} {\mathsf e}_{\mathscr C}>0$, see \cite[Theorem 8.15 p. 263]{McL00}).
\end{definition}
%\begin{prop}[{\cite[p. 219]{McL00}}]
%	If $\mathscr C$ is $C^{1,1}$, then ensures
%	that ${\mathsf e}_{\mathscr C}\in H^{\frac{1}{2}}({\mathscr C})$.
%\end{prop}
Setting 
$$
H({\mathscr C}):=\{q\in H^{\frac{1}{2}}({\mathscr C})\,:\,\langle \hat{\mathsf e}_{\mathscr C}, q\ranglemp{{\mathscr C}} =0\}
$$
we know, following McLean \cite{McL00}, that the linear operator:
$$ 
{\mathsf S}_{\mathscr C}:  \hat q\in {\widehat H}({\mathscr C})\longmapsto q\in
{H}({\mathscr C}),$$
defines an isomorphism that extends into an isomorphism from $H^{-\frac{1}{2}}({\mathscr C})$ onto
$H^{\frac{1}{2}}({\mathscr C})$,  if ${\rm Cap}({\mathscr C})\neq 1$ (see \cite[Theorem 8.16]{McL00}).
Under this condition, one can identify via this isomorphism any density $\hat q \in  {\widehat H}({\mathscr C})$ with the trace 
$$q:={\mathsf S}_{\mathscr C}  \hat q \in H({\mathscr C}).$$
This identification will be systematically used throughout the paper, using the notation with (respectively without) a hat on the density like quantities (respectively on traces). This isomorphism turns out to be an isometry provided the  spaces ${\widehat H}({\mathscr C})$ and ${H}({\mathscr C})$ are endowed with the following inner products:
\begin{definition}
	\label{inner} 
	For all $\hat q,\hat p\in {\widehat H}({\mathscr C})$, we set:
	$$\langle \hat q,\hat p\ranglemm{{\mathscr C}}=\langle q,p\ranglepp{{\mathscr C}}=\langle \hat q, p \ranglemp{{\mathscr C}}.$$
%	For all $q,p\in {H}({\mathscr C})$, we denote by $\hat q= {\mathsf S}_{\mathscr C}^{-1}q$	and $\hat p= {\mathsf S}_{\mathscr C}^{-1}p$ the corresponding densities and we set:
%	$$\langle q,p\ranglepp{{\mathscr C}}:=\langle \hat q,\hat q\ranglemm{{\mathscr C}}=\langle  \hat q,p\ranglemp{{\mathscr C}}.$$
\end{definition}
%Note that with the norms associated with these inner products, $ {\mathsf S}_{\mathscr C}$
%is now an isometry. Using this isometric isomorphism to identify $H({\mathscr
%	C})$ and ${\widehat H}({\mathscr C})$, we will use henceforth the
%notation $\hat f$ to denote the element ${\mathsf S}_{\mathscr C}^{-1}f  \in {\widehat
%	H}({\mathscr C})$, for every $f\in H({\mathscr C})$.

%The space ${\widehat H}({\mathscr C})$ could have been equivalently defined as
%the space 
%of  densities $\hat f$ for which the Dirichlet energy of the corresponding
%single layer potential is finite, i.e.
%\begin{equation}
%\label{equiv:decroi}
%\int_{\mathbb R^2}|\nabla\mathscr S_{\mathscr C} \hat f|^2\,{\rm
%	d}x<+\infty.
%	\end{equation}
%We have the equivalence, for every $q\in \Hp{\curve}$:
%	\begin{equation}
%	\label{asymp}
%	q\in H(\curve)\Leftrightarrow \SL{\curve} \hat q(x)=o(|G(x)|)\quad\text{as}\quad |x|\to+\infty.
%	\end{equation}
According to \eqref{equiv:decroi}, the inner products introduced in Definition~\ref{inner} are  related to
the Dirichlet energy of the single layer potential through the following
identities:
$$\|\hat q\|^2_{-\frac{1}{2},\mathscr C}=\|q\|^2_{\frac{1}{2},\mathscr 	C}=\int_{\mathbb R^2}|\nabla(\mathscr S_{\mathscr C} \hat q)|^2\,{\rm d}x,\qquad \forall \hat q\in {\widehat H}({\mathscr C}).
$$
We also need in the sequel the following orthogonal projections.
\begin{definition}
Let $\Pi_{\mathscr C}$ and $\widehat\Pi_{\mathscr C}$ denote respectively the orthogonal projections  from $H^{\frac{1}{2}}({\mathscr C})$ into $H({\mathscr C})$ and from $H^{-\frac{1}{2}}({\mathscr C})$ into ${\widehat H}({\mathscr C})$.  
\end{definition}
In particular, we have the following unique decompositions:
\begin{alignat*}{3}
\forall\,  \hat q\in H^{-\frac{1}{2}}({\mathscr C}):\qquad \hat q &= \langle \hat q,1\rangle \,\hat{\mathsf e}_{\mathscr C} + \hat q_0,&\qquad&  \hat q_{0}&:=\widehat \Pi_{\mathscr C}\hat q\in {\widehat H}({\mathscr C}),\\
\forall \,  q\in H^{\frac{1}{2}}({{\mathscr C}}):\qquad q &= \langle  \hat {\mathsf e}_{\mathscr C}, q\rangle \, 1 +  q_0,&&  q_0&:=\Pi_{\mathscr C} q\in H({{\mathscr C}}).
\end{alignat*}
\begin{definition}\label{def:Tr0}
We denote by ${\mathsf{Tr}}_{\mathscr C}$  the classical trace operator (valued into $\Hp{\curve}$), and by ${\mathsf{Tr}}^0_{\mathscr C}$ when it is left-composed with the orthogonal projection onto $H(\mathscr C)$: ${\mathsf{Tr}}^0_{\mathscr C}:= \Pi_\curve {\mathsf{Tr}}_{\mathscr C}$.
\end{definition}
Let us conclude this preliminary section by a useful characterization of the chosen norm on $H(\mathscr C)$. Classically, we define the quotient weighted Sobolev space:
$$W^1_0(\mathbb R^2)=\{u\in\mathcal D'(\mathbb R^2)\,:\,\rho u \in L^2(\mathbb R^2),\,\nabla u\in (L^2(\mathbb R^2))^2\}/\mathbb R,$$
where the weight is given by 
$$\rho(x):=\left(\sqrt{1+|x|^2}\log(2+|x|^2)\right)^{-1},\qquad x\in\mathbb R^2,$$
and where the quotient means that functions of $W^1_0(\mathbb R^2)$ are defined up to an additive constant. This space is a Hilbert space once equipped with the inner product:
$$\langle u,v\rangle_{W^1_0(\mathbb R^2)}:=\int_{\mathbb R^2}\nabla u\cdot\nabla v\,{\rm d}x.$$
In particular, according to \eqref{equiv:decroi},  ${\mathscr S}_{\mathscr C}\hat q\in W^1_0(\mathbb R^2)$ if and only if  $q\in H({\mathscr C})$, and moreover 
$$
\|\hat q\|_{-\frac{1}{2},\mathscr C} = \|q\|_{\frac{1}{2},\mathscr C} = \|{\mathscr S}_{\mathscr C}\hat q\|_{W^1_0(\mathbb R^2)}, \qquad \forall q\in H({\mathscr C}).
$$
%-------------------------------------
\begin{lemma}
\label{LEM:min}
For every $q\in H(\mathscr C)$, we have
$$\|q\|_{\frac{1}{2},\mathscr C}=\inf\left\{\|u\|_{W^1_0(\mathbb R^2)}\,:\,u\in W^1_0(\mathbb R^2) \text{ and } {\mathsf Tr}^0_{\mathscr C}u=q\right\}.$$
The infimum is a minimum which is uniquely achieved by $u=\mathscr S_{\mathscr C}\hat q$.
\end{lemma}
%Notice that $W^1_0(\mathbb R^2)$ consists of class of functions, two representatives of the same class differing up to an additive constant. The purpose of the projected trace operator in the lemma (instead of the usual trace operator) is to fix this constant, so that, in 
%the identity ${\mathsf Tr}^0_{\mathscr C}u=q$, $u$ can be taken to be any representative of the class. 

\begin{proof}
Given $q\in H(\mathscr C)$, let us consider the orthogonal decomposition:
$$W^1_0(\mathbb R^2)=\langle  \mathscr S_{\mathscr C}\hat q\rangle\oplus \langle  \mathscr S_{\mathscr C}\hat q\rangle^\perp.$$
Let $u\in W^1_0(\mathbb R^2)$ be such that ${\mathsf Tr}^0_{\mathscr C}u=q$. Writing  $u$ in the form $u=\lambda \mathscr S_{\mathscr C}\hat q+v$ with $\lambda\in\mathbb R$ and $v\in \langle\mathscr S_{\mathscr C}\hat q\rangle^\perp$, and taking the projected trace on $\mathscr C$ we get:
$$q=\lambda q+{\mathsf Tr}^0_{\mathscr C} v\qquad\text{in }H(\mathscr C).$$
Forming now the duality product with $\hat q$ and taking into account that: $\langle \hat q, {\mathsf Tr}^0_{\mathscr C} v\ranglemp{\mathscr C}=\langle \mathscr S_{\mathscr C}\hat q,v\rangle_{W^1_0(\mathbb R^2)}=0$, we deduce that $\lambda = 1$. Since we have now $$\|u\|^2_{W^1_0(\mathbb R^2)}=\|q\|_{\frac{1}{2},\mathscr C}^2+\|v\|^2_{W^1_0(\mathbb R^2)},$$
with $v$ such that ${\mathsf Tr}^0_{\mathscr C}v=0$, the conclusion follows.
\end{proof}

\subsection{Boundary interaction and single layer potential}
%================================================================

In this section, we are interested in quantifying the Dirichlet energy variation between $\SL{\Gamma}\hat q$ and $\SL{\gamma}\hat p$ where $p={\mathsf Tr}^0_\gamma \SL{\Gamma}\hat q$ (i.e. $p$ is the trace of 
the single layer potential $\SL{\Gamma}\hat q$ on $\gamma$).
%We set  $\mathsf C_i=\Omega\setminus\omega$ and $\mathsf C_e=\mathbb
%R^2\setminus\overline{\mathsf C}_i$, so that $\mathsf C_e$ consists of two connected
%components: an unbounded one, denoted by $\mathsf C_e^+$ and a bounded one,
%denoted by $\mathsf C_e^-$.

\begin{definition}
\label{def_K}
%Given ${\mathscr C}$ and ${\mathscr D}$, the boundaries of two Lipschitz domains such that $\curve\cap\mathscr D=\varnothing$, 
We define the boundary interaction operators  $\mathsf K_{\Gamma}^{\gamma}$ and $\mathsf K_{\gamma}^{\Gamma}$ between $\Gamma$ and $\gamma$  by:
$$
\mathsf K_{\Gamma}^{\gamma}: q\in H(\Gamma)\longmapsto {\mathsf{Tr}}^0_{\gamma}(\mathscr S_{\Gamma} {\hat q})\in H(\gamma),
\qquad
\mathsf K_{\gamma}^{\Gamma}: p\in H(\gamma)\longmapsto {\mathsf{Tr}}^0_{\Gamma}(\mathscr S_{\gamma} {\hat p})\in H(\Gamma),
$$
where ${\mathsf{Tr}}^0_{\gamma}$ and ${\mathsf{Tr}}^0_{\Gamma}$ are given in Definition~\ref{def:Tr0}.
\end{definition}

%--------------------------------------------------------------------
The next result shows that ${\mathsf{Tr}}^0_{\Gamma}$ can be replaced by ${\mathsf{Tr}}_{\Gamma}$ in the definition of $\mathsf K_{\gamma}^{\Gamma}$:
\begin{lemma}
\label{porj_nomore}
If $p\in H(\gamma)$, then $q:={\mathsf Tr}_\Gamma(\SL{\gamma}\hat p)$ belongs to $H(\Gamma)$. 
\end{lemma}
%--------------------------------------------------------------------
\begin{proof}
Let $p$ and $q$ be given as in the statement of the lemma and let us define the function $w:=(w_i,w_e)$ in $H^1_{\ell oc}(\mathbb R^2)$ by setting: $w_i=\SL{\Gamma}\hat q$ in $\Omega$ and $w_e=\SL{\gamma}\hat p$ in $\mathbb R^2\setminus\overline{\Omega}$. 
According to \eqref{equiv:decroi}, the function $w$ has finite Dirichlet energy since $p\in H(\gamma)$.  Thanks to \eqref{asymptSC}, we see that $w_e(x)=O(|x|^{-1})$ at infinity, and this allows us to obtain the following classical integral representation formula for every $x\in \mathbb R^2\setminus\overline{\Omega}$ (see for instance \cite[p.~182]{Ste08b} or \cite[Lemma 3.5]{CosSte85})
\begin{align*}
w_e(x)&=\langle  \partial_n G(x-\cdot), w_e  \ranglepp{\Gamma}
- \langle \partial_n w_e, G(x-\cdot)\ranglemp{\Gamma} \\
0&=\langle \partial_n G(x-\cdot), w_i \ranglemp{\Gamma} 
- \langle \partial_n w_i, G(x-\cdot)\ranglemp{\Gamma}.
\end{align*}
Since ${\mathsf Tr}_\Gamma w_i={\mathsf Tr}_\Gamma w_e=q$, we get by subtracting these identities that:
$$w_e(x)=\langle \hat r,G(x-\cdot)\ranglemp{\Gamma} =\SL{\Gamma}\hat r(x),\qquad x\in\mathbb R^2\setminus\overline{\Omega},$$
where the density $\hat r:=\partial_n w_i-\partial_n w_e$ belongs to $\widehat H(\Gamma)$ since, as already mentioned, $w$ has finite Dirichlet energy. Taking the trace on  $\Gamma$, we deduce from the above relation that $q= r\in H(\Gamma)$ and the proof is complete.
\end{proof}
%-------------------------------------
\begin{prop}
\label{prop:symK}
The operators $\mathsf K_\Gamma^\gamma$ and $\mathsf K_\gamma^\Gamma$ are compact, one-to-one and dense-range operators. Moreover, for every functions $q\in H(\Gamma)$ and $p\in H(\gamma)$, we have:
$$\langle \mathsf K_\Gamma^{\gamma} q,p\ranglepp{\gamma} = \langle  q,\mathsf K^\Gamma_\gamma p\ranglepp{\Gamma}.$$
\end{prop}
\begin{proof}
The compactness follows from the regularity of the single layer potential away from the boundary, combined to \cite[Proposition 13.5.8]{Tucsnak:2009aa}. 

Addressing the symmetry property, consider $q\in L^\infty(\Gamma)\cap H(\Gamma)$ and $p\in L^\infty(\gamma)\cap H(\gamma)$. We can write that:
$$\langle \mathsf K_\Gamma^\gamma q,p\ranglepp{\gamma}=\int_\gamma\int_\Gamma G(x-y)\hat q(y)\hat p(x)\,{\rm d}\sigma_y{\rm d}\sigma_x=\int_\Gamma\int_\gamma G(x-y)\hat q(x)\hat p(y)\,{\rm d}\sigma_y{\rm d}\sigma_x
= \langle \mathsf K_\gamma^\Gamma p,q\ranglepp{\Gamma},$$
and the conclusion follows by density. 

Assume now that $q\in H(\Gamma)$ is such that $ \mathsf K_\Gamma^\gamma q=0$. By the unique continuation property for harmonic functions, it means that $\mathscr S_\Gamma\hat q$ 
is constant in $\Omega$ and hence that $q=0$. Since 
$$\overline{{\rm Ran}\,\mathsf K_\gamma^\Gamma}={\rm Ker} \, \mathsf K_\Gamma^\gamma,$$
we get the density result and the proof is completed.
\end{proof}

\begin{prop}
\label{prop:contract_K}
The norms of the operators  $\mathsf K_\Gamma^\gamma$ and $\mathsf K_\gamma^\Gamma$ are strictly less that 1.
\end{prop}
\begin{proof}
According to Lemma~\ref{LEM:min}  we have, for every $q\in H(\Gamma)$:
$$\|\mathsf K_{\Gamma}^{\gamma} q\|_{\frac{1}{2},\gamma}=\inf\left\{\|u\|_{W^1_0(\mathbb R^2)}\,:\,u\in W^1_0(\mathbb R^2) \text{ and } {\mathsf Tr}^0_{\gamma}u={\mathsf Tr}^0_{\gamma}(\mathscr S_{\Gamma} \hat q)\right\}.$$
We deduce that $\|\mathsf K_{\Gamma}^{\gamma} q\|_{\frac{1}{2},\gamma}\leqslant \|\mathscr S_{\Gamma} \hat q\|_{W^1_0(\mathbb R^2)}=\|q\|_{\frac{1}{2},\Gamma}$ and the norm of $\mathsf K_{\Gamma}^{\gamma}$ 
is no greater than 1. 

The operator $\mathsf K_{\Gamma}^{\gamma}$ being compact, its norm is achieved by some $q_{\Gamma}\in H(\Gamma)$. If 
$\| \mathsf K_{\Gamma}^{\gamma} q_{\Gamma}\|_{\frac{1}{2},\gamma}=\|q_{\Gamma}\|_{\frac{1}{2},\Gamma}$, we would have, according to Lemma~\ref{LEM:min}:
$$
\mathscr S_{\gamma} \hat q_{\gamma}=\mathscr S_{\Gamma} \hat q_{\Gamma},\qquad \text{in }\mathbb R^{2},
$$
where $q_{\gamma}:=\mathsf K_{\Gamma}^{\gamma} q_\Gamma$. 
This identity implies that $\hat q_{\Gamma} = [\partial_n( \mathscr S_{\gamma} \hat q_{\gamma})]_{\Gamma}=0$, yielding the expected contradiction.
%Define now:
%\begin{equation}
%\label{def:min}
%k_{\Gamma}^{\gamma}:=\sup\{\|\mathsf K_{\Gamma}^{\gamma} q\|_{\frac{1}{2},\gamma}\,:\, q\in H(\Gamma),\,\|q\|_{\frac{1}{2},\Gamma}= 1\},
%\end{equation}
%and consider $(q_n)_n$, a minimizing sequence, which can be assumed to be weakly convergent in $H({\Gamma})$. The operator $\mathsf K_{\Gamma}^{\gamma}$ being compact, the supremum 
%in \eqref{def:min} is actually a maximum achieved by some $q^\ast\in H(\Gamma)$. If $k_{\Gamma}^{\gamma}=1$, one would have $\| \mathsf K_{\Gamma}^{\gamma} q^\ast\|_{\frac{1}{2},\gamma}=\|q^\ast\|_{\frac{1}{2},\Gamma}$ and then, according to Lemma~\ref{LEM:min}:
%$$
%\mathscr S_{\gamma}\widehat{\mathsf K_{\Gamma}^{\gamma} q^\ast}=\mathscr S_{\Gamma} \hat q^\ast,
%$$
%where we have set, with obvious notation, $\widehat{\mathsf K_{\Gamma}^{\gamma} q^\ast}:=\mathsf S_{\gamma}^{-1}\mathsf K_{\Gamma}^{\gamma} q$. This identity means that $\mathscr S_{\Gamma} \hat q^\ast$ is harmonic in $\mathbb R^2$ and therefore that $q^\ast=0$, yielding the desired contradiction.
\end{proof}

\subsection{Integral formulation and well-posedness}
%================================================================
Let us go back to Problem \eqref{main_problem}. Without loss of generality, let us assume from now on that  the diameter of $\Omega$ is less than 1  (otherwise, it suffices to rescale the problem), which implies in particular that ${\rm Cap}(\Gamma)< 1$ and ${\rm Cap}(\gamma)< 1$ (see \cite[p.~143]{Ste08b} and references therein). 

\begin{prop}
\label{prop:RI_SC}
For every $f\in H(\Gamma)$, denote by $(u^f,c^f)\in H^1(\Omega\setminus\overline{\omega})\times\mathbb R$ the unique solution of System \eqref{main_problem}. The function $u^f$  can be
represented as a superposition of single layer potentials as follows:
\begin{equation}
\label{def_uf}
u^f=\mathscr S_{\gamma}\hat p+\mathscr S_\Gamma \hat q,
\end{equation}
where $\hat p\in \widehat H(\gamma)$ and $\hat q\in H^{-\frac{1}{2}}(\Gamma)$ solve the
following system of coupled integral equations on the boundaries $\gamma$ and
$\Gamma$:
\begin{subequations}
\label{eq:EIpq}
\begin{alignat}{3}
 p+\mathsf{Tr}_{\gamma}(\mathscr S_\Gamma \hat q)&=c^f&\qquad&\text{on
}\gamma \label{eq:EIpq_11} \\
\mathsf{Tr}_{\Gamma}(\mathscr S_{\gamma}\hat p)+ q&=f&&\text{on
}\Gamma.\label{eq:EIpq_12}
\end{alignat}
\end{subequations}
\end{prop}

\begin{proof}
It is a  consequence of \cite[Theorem 8.16]{McL00}
that the unique solution to System \eqref{main_problem} can be written as 
a superposition of two single layer potentials respectively supported on $\Gamma$ and $\gamma$ and respectively associated with the densities $(\hat p,\hat q)\in H^{-\frac{1}{2}}(\gamma)\times H^{-\frac{1}{2}}(\Gamma)$, as in \eqref{def_uf}. 
It only remains to verify that $p$ is in fact in  $\widehat H(\gamma)$, i.e. that
$\langle \hat p,1\ranglemp{\gamma}=0$. This is a straightforward consequence of \eqref{free_circ} and the jump relation for the normal derivative of the single layer potential. 
\end{proof}
%================================================================
\section{The reconstruction formula}
\label{sect:MAIN}
%================================================================

Going back to the DtN operator $\Lambda_\gamma$ of problem \eqref{main_problem}, and due to \eqref{free_circ}, we have by Green's formula
$$
\langle \partial_n u^f, 1\ranglemp{{\Gamma}} = -\langle \partial_n u^f, 1\ranglemp{{\gamma}}  =0,
$$
which shows that  $\Lambda_\gamma$ is valued in ${\widehat H}({\Gamma})$. Considering data $f\in {H}({\Gamma})$, we can thus define the DtN operator $\Lambda_\gamma$ as follows:
\begin{equation}
\label{D_to_N}
\Lambda_\gamma:f\in H(\Gamma)\longmapsto \partial_nu^f\in {\widehat
	H}({\Gamma}).
\end{equation}
In the case where  $\omega=\varnothing$, we will denote respectively by $u_0^f$ and $\Lambda_0$ the solution $u^f$ and the DtN $\Lambda_\gamma$. Note that we have in particular 
$$u^f_0=\mathscr S_\Gamma \hat f.$$

\subsection{Factorization of the DtN map}
%================================================================

\begin{theorem}
\label{first_decom}
The two following bounded linear operators in $H(\Gamma)$:
$$
\mathsf R:=\mathsf S_\Gamma(\Lambda_\gamma-\Lambda_0)\qquad\text{and}\qquad \mathsf K:=\mathsf K_\gamma^\Gamma\mathsf K_\Gamma^\gamma,
$$
satisfy the following equivalent identities:
\begin{equation}
\label{factor:main}
\mathsf R=({\rm Id}-\mathsf K)^{-1}{\mathsf K},
\qquad\qquad
\mathsf K=({\rm Id}+\mathsf R)^{-1}\mathsf R.
\end{equation}
\end{theorem}

\begin{proof} Given $f$ in $H(\Gamma)$, let $(\hat p,\hat q)\in \widehat H(\gamma)\times \Hm{\Gamma}$ be the solution of System \eqref{eq:EIpq}.
% as:
%\begin{subequations}
%\label{eq:EIpq_1}
%\begin{alignat}{3}
%p+\mathsf{Tr}_{\gamma}(\mathscr S_\Gamma \hat q)&=c^f&\qquad&\text{on
%}\gamma \label{eq:EIpq_11}\\
%\mathsf{Tr}_{\Gamma}(\mathscr S_{\gamma}\hat p)+  q&=f&&\text{on
%}\Gamma,\label{eq:EIpq_12}
%\end{alignat}
%\end{subequations}
According to Lemma~\ref{porj_nomore}, $\mathsf{Tr}_{\Gamma}(\mathscr S_{\gamma}\hat p)\in  H(\Gamma)$ and hence $\mathsf{Tr}_{\Gamma}(\mathscr S_{\gamma}\hat p)={\mathsf K}_\gamma^\Gamma p$. Since $f\in H(\Gamma)$, we deduce from \eqref{eq:EIpq_12} 
that $q=f-{\mathsf K}_\gamma^\Gamma p\in H(\Gamma)$. Applying the projector $\Pi_\gamma$ to \eqref{eq:EIpq_11}, we obtain the following system:
\begin{subequations}
\begin{alignat}{3}
p+ {\mathsf K}_\Gamma^\gamma q&=0&\qquad&\text{on
}\gamma \label{eq:EIpq_21}\\
{\mathsf K}_\gamma^\Gamma p+  q&=f&&\text{on
}\Gamma.\label{eq:EIpq_22}
\end{alignat}
\end{subequations}
Eliminating $p$, it follows that $({\rm Id}-\mathsf K)q=f$ and hence $({\rm Id}-\mathsf K)(q-f)={\mathsf K}f$. The operator $\mathsf K$ being a contraction (see Propositon~\ref{prop:contract_K}), we end up with:
\begin{equation}
\label{lahaut}
q-f=({\rm Id}-\mathsf K)^{-1}{\mathsf K}f.
\end{equation}
On the other hand, we have
\begin{equation}\label{eq:diffDtn}
(\Lambda_\gamma - \Lambda_0)f =\partial_{n}(\mathscr S_\Gamma \hat q + \mathscr S_\gamma \hat p)|_\Gamma -\partial_{n}(\mathscr S_\Gamma \hat f)|_\Gamma.
\end{equation}
But outside $\Gamma$, the two single layer potentials $\mathscr S_\Gamma \hat f$ and $\mathscr S_\Gamma \hat q + \mathscr S_\gamma \hat p$  both solve the well-posed Dirichlet exterior boundary value problem:
\begin{alignat*}{3}
-\Delta u&=0&\quad&\quad\text{in }\mathbb R^2\setminus \overline{\Omega},\\
u&=f&&\quad\text{on }\Gamma,\\
u&=O(|x|^{-1})&&\quad |x|\to+\infty.
\end{alignat*}
Hence $ (\mathscr S_\Gamma \hat q + \mathscr S_\gamma \hat p) = \mathscr S_\Gamma \hat f$ in $\mathbb R^2\setminus\overline{\Omega}$, and in particular we can rewrite \eqref{eq:diffDtn} as
$$
(\Lambda_\gamma - \Lambda_0)f = \left[ \partial_n(\mathscr S_\Gamma \hat q + \mathscr S_\gamma \hat p)\right]_\Gamma -[ \partial_n(\mathscr S_\Gamma \hat f)]_\Gamma  =  \hat q -\hat f,
$$
where the last equality follows from the jump relation \eqref{eq:jump}. Comparing this relation and \eqref{lahaut}, we obtain that $\mathsf S_\Gamma(\Lambda_\gamma - \Lambda_0)f = ({\rm Id}-\mathsf K)^{-1}{\mathsf K}f$, which is exactly the first relation in \eqref{factor:main}. The second relation follows easily. 
\end{proof}
The first equation in \eqref{factor:main} can be seen as a factorization of the (known) DtN operator $\Lambda_\gamma-\Lambda_0$ in terms of the (unknown) boundary interaction operator $\mathsf K_\Gamma^\gamma$ and $\mathsf K_\gamma^\Gamma$. Similarly, the second equation in \eqref{factor:main} can be seen as a factorization of the boundary interaction operator $\mathsf K=\mathsf K_\gamma^\Gamma\mathsf K_\Gamma^\gamma$ in terms of the measurement operator $\mathsf R$ (which is entirely determined by the perturbed and unperturbed DtN maps and by the exterior boundary $\Gamma$). Using Proposition~\ref{prop:symK}, it is worth reformulating this second equation in a variational form:
\begin{equation}
\label{factor:main_varia}
\langle \mathsf K_\Gamma^\gamma f,\mathsf K_\Gamma^\gamma g\ranglepp{\gamma} = \langle ({\rm Id}+\mathsf R)^{-1}\mathsf R f,g\ranglepp{\Gamma},\qquad\forall\,f,g\in H(\Gamma).
\end{equation}
This identity constitutes the first step towards the reconstruction of the unknown boundary $\gamma$. Indeed, the bilinear form $\langle \mathsf K_\Gamma^\gamma \,\cdot,\mathsf K_\Gamma^\gamma \,\cdot\ranglepp{\gamma}$ turns out to encode the geometry of the inclusion, as shown in the next section.

\subsection{Harmonic polynomials and GPST}
%================================================================
Throughout the paper, we identify $x=(x_{1},x_{2})$ in $\mathbb R^2$ with the complex number $z=x_1+ix_2$.
\begin{definition}\label{def:PQ}
%When $N=2$ and 
For every $m\geqslant 1$, we define the harmonic polynomials of degree $m$:
$$P_1^m(x)=\Re\left(z^m\right)\qquad\text{and}\qquad P_2^m(x)=\Im\left(z^m\right).$$
%When $N=3$, there are $2m+1$ independent harmonic polynomials of degree $m$:
%$$P_\ell^m(x)=\rho^m Y_\ell^m(\theta,\phi),\qquad \ell=1,\ldots,2m+1,$$
%where $(\rho,\theta,\phi)$ are the spherical coordinates and $Y_\ell^m$ is the spherical harmonic function of degree $\ell$ and order $m$ (\AM{REF ICI}).
%For every $r\in \mathbb R^N$, we define:
We define as well
\begin{equation}\label{eq:Qm12}
Q_{1,\Gamma}^m(x):=P_\ell^m(x) + c_{1,\Gamma}^m\qquad\text{and}\qquad 
Q_{2,\Gamma}^m(x):=P_\ell^m(x) + c_{2,\Gamma}^m
\end{equation}
where the constant $c_{\ell,\Gamma}^m\in\mathbb R$, $\ell=1,2$, are chosen such that the trace of $Q_{\ell,\Gamma}^m$ on $\Gamma$ belongs to $H(\Gamma)$. \\
Finally, we set
\begin{equation}\label{eq:Qm}
Q_{\Gamma}^m:= Q_{1,\Gamma}^m+i Q_{2,\Gamma}^m.
\end{equation}
\end{definition}
%-------------------------------------------------------------------
The crucial point about these polynomials $Q_{\ell,\Gamma}^m$, $\ell=1,2$, lies in the fact that since they are harmonic, we have
$$
\mathsf K_{\Gamma}^{\gamma}(Q_{\ell,\Gamma}^m\big|_{\Gamma}) = Q_{\ell,\gamma}^m\big|_{\gamma},
$$
and hence, using these harmonic polynomials $Q_{\ell,\Gamma}^m$ in formula \eqref{factor:main_varia} (and using for simplicity the same notation for the functions and their traces on the boundaries $\gamma$ and $\Gamma$), we obtain that for all $m,m'\geqslant 1$ and all $\ell,\ell'=1,2$:
$$\langle \mathsf K_\Gamma^\gamma Q_{\ell,\Gamma}^m,\mathsf K_\Gamma^\gamma Q^{m'}_{\ell',\Gamma}\ranglepp{\Gamma}=\langle Q_{\ell,\gamma}^m, Q^{m'}_{\ell',\gamma}\ranglepp{\gamma}.
$$
\begin{rem}
The quantities $\langle Q^m_{\ell,\gamma}, Q^{m'}_{\ell',\gamma} \rangle_{\frac 12, \gamma}$ are strongly connected with the so-called Generalized P\'olya-Szeg\"o Tensors (GPST)  appearing in the high-order asymptotic expansion of the DtN map for small inclusions (see for instance the recent papers by Ammari {\it et al.} \cite{AmmGarKan14,AmmDenKan14} and references therein). Our definition is somehow different from theirs, as they use real polynomials $x^m$, while we use harmonic polynomials.
\end{rem}
%==============================================================================
\subsection{From the GPST to the geometry of the cavity: an explicit inversion formula}
In this section, we are going to see that the quantities $\langle Q_{\ell,\gamma}^m, Q^{m'}_{\ell',\gamma}\ranglepp{\gamma}$ for $m,m'\geqslant 1$ and $\ell,\ell'=1,2$, which can be deduced from the measurements (see \eqref{factor:main_varia}), contain all the necessary information to reconstruct the cavity. We can even say more: the geometric information of $\gamma$ is actually redundant in the GPST. As we shall see, the knowledge of the quantities $\langle Q_{\gamma}^m, Q^{1}_{\gamma}\ranglepp{\gamma}$ and $\langle Q_{\gamma}^m, \overline {Q^{1}_{\gamma}}\ranglepp{\gamma}$ suffices to reconstruct the cavity. More precisely, assume that the geometry of $\gamma$ is described through the conformal mapping 
$$
\phi:z\mapsto a_1 z+a_0+\sum_{m\leqslant -1} a_m z^m,
$$
that maps the exterior of the unit disk $D$ onto the exterior of $\omega$ (see the book of Pommerenke \cite[p.~5]{Pommerenke:1992aa} for the existence of such a mapping). In particular, $t\in]-\pi,\pi]\mapsto \phi(e^{it})$ provides a parameterization of $\gamma$. Notice that in this description, $|a_1|$ is the logarithmic capacity of $\gamma$ and can be chosen such that $a_1>0$. The coefficient $a_0$ is  the conformal center of $\omega$. With these notation, we have the following result. 

\begin{theorem}\label{thm:expl}
Let $(Q_{\gamma}^m)_{m\geqslant 1}$ be the complex harmonic polynomials defined by \eqref{eq:Qm12}-\eqref{eq:Qm}. Define the two following sequences of complex numbers ($1\leqslant m\leqslant+\infty$):
\begin{subequations}
\label{la_haut}
\begin{equation}\label{eq:mum}
\mu_{m}:=\frac12 \langle Q_\gamma^m, \overline{Q^{1}_\gamma}\ranglepp{\gamma} =\frac{1}{2}\langle Q^m_\Gamma,({\rm Id}+\mathsf R)^{-1}\mathsf R\overline {Q^1_\Gamma}\ranglepp{\Gamma},
\end{equation}
\begin{equation}\label{eq:num}
\nu_{m}:= \frac12\langle Q_\gamma^m, Q^{1}_\gamma\ranglepp{\gamma}=\frac{1}{2}\langle Q^m_\Gamma, ({\rm Id}+\mathsf R)^{-1}\mathsf R Q^1_\Gamma\ranglepp{\Gamma},
\end{equation}
with $\mathsf R:=\mathsf S_\Gamma(\Lambda_\gamma-\Lambda_0)$. Then, $\mu_1>0$ and we have the explicit formulae:
\begin{align}
a_1&=\left(\frac{\mu_1}{2\pi}\right)^{\frac{1}{2}}\qquad \qquad \qquad a_0=\frac{\mu_2}{2\mu_1}\\
\quad a_{-m}&=\mu_1^{-\frac{m}{2}}\sum_{\alpha\in\mathsf A_m} C_\alpha
\left(\frac{\mu_2}{\mu_1}\right)^{\alpha_0}\nu_{1}^{\alpha_1}\nu_{2}^{\alpha_2}\ldots \nu_m^{\alpha_m},\quad m\geqslant 1,
\end{align}
where 
\begin{equation}\label{eq:An}
\mathsf A_m:=\{\alpha\in \mathbb N^{m+1}\,:\,\alpha_0+2\alpha_1+3\alpha_2+\ldots+(m+1)\alpha_m=(m+1), \,\alpha_{0}\neq m+1\}
\end{equation}
and 
\begin{equation}\label{eq:Calpha}
C_\alpha:= \frac{(-1)^{|\alpha|+1}}{2^{\alpha_0}m} 
\frac{(2\pi)^{\frac{m}{2} - (\alpha_{1}+\dots+\alpha_{m}) }}{{1^{\alpha_1}2^{\alpha_2}\ldots m^{\alpha_m}}}.
\end{equation}
\end{subequations}
\end{theorem}
The rest of this section is devoted to the proof of this result. 

To simplify the forthcoming computation, we complete the sequence of complex numbers $(a_k)_{k\leqslant 1}$ by setting $a_k=0$ for $k\geqslant 2$. We denote $a^n_k$ ($n\in\mathbb N$, $k\in\mathbb Z$) the $k$th coefficients of the Laurent's series of $\phi^n$:
\begin{equation}
\label{def_an}
a_k^n=\sum_{|\alpha|=k}a_{\alpha_1}a_{\alpha_2}\ldots a_{\alpha_n},
\end{equation}
where the sum ranges over all the multi-indices $\alpha=(\alpha_1,\ldots,\alpha_n)\in\mathbb Z^n$ whose length $|\alpha|=\alpha_1+\ldots+\alpha_n$ is equal to $k$.
We also introduce the quantities:
$$
\mu^{m,m'} :=\frac12\langle Q_{\gamma}^m, \overline{ Q^{m'}_{\gamma}}\ranglepp{\gamma}
=\frac{1}{2}\langle Q^m_\Gamma,({\rm Id}+\mathsf R)^{-1}\mathsf R\overline {Q^{m'}_\Gamma}\ranglepp{\Gamma}$$
$$
\nu^{m,m'} :=\frac12\langle Q_{\gamma}^m, Q^{m'}_{\gamma}\ranglepp{\gamma}= \frac{1}{2}\langle Q^m_{\Gamma}, ({\rm Id}+\mathsf R)^{-1}\mathsf R Q^{m'}_{\Gamma}\ranglepp{\Gamma},
$$
so that
$$
\mu_{m}=\mu^{m,1} \qquad\text{and}\qquad
\nu_{m}=\nu^{m,1} .
$$

%------------------------------------------------------------------------------------------------------------------------------
\begin{lemma}
Denoting, for every $m\geqslant 1$:
$$\phi^m_+(z)=\sum_{k\geqslant 1}a_k^m z^k\qquad\text{and}\qquad \phi^m_-(z)=\sum_{k\leqslant -1}a_k^m z^k,$$
the following identities hold true:
\begin{equation}
\label{main_mu_nu}
\mu^{m,m'}=\int_{-\pi}^\pi
\overline{e^{it}(\phi_{+}^{m'})'(e^{it})}\phi^{m}(e^{it})\,{\rm
	d}t\quad\text{and}\quad
\nu^{m,m'}=\int_{-\pi}^\pi
e^{it}(\phi_{+}^{m'})'(e^{it})\phi^{m}(e^{it})\,{\rm d}t.
\end{equation}
\end{lemma}

%------------------------------------------------------------------------------------------------------------------------------
\begin{proof} Let $m,m'\geqslant 1$ and $\ell=1,2$ be fixed. For the sake of simplicity, we drop in this proof the dependence with respect to $\gamma$ and we denote $Q_{\ell,\gamma}^m$ simply by $Q_\ell^m$). 
We aim to compute the quantity:
$$
\langle Q_\ell^m,Q_\ell^{m'}\ranglepp{\gamma}=\langle \widehat Q_\ell^m , Q_\ell^{m'} \ranglemp{{\gamma}}.
$$
To do so, we recall that from the jump relation \eqref{eq:jump}, we have $\displaystyle\widehat Q_\ell^m = [ \partial_n U^m_\ell ]_{\gamma}$, where $U^m_\ell:= \mathscr S_{\gamma} \widehat Q_\ell^m$. Let us denote  by $U^m_{e,\ell}$ and $U^m_{i,\ell}$ the restrictions of $U^m_\ell$ respectively to $\mathbb R^2\setminus\overline\omega$ and $\omega$.

We know that $U^m_{e,\ell}$  solves the following exterior Dirichlet boundary problem:
\begin{subequations}
\label{ext_prb}
\begin{alignat}{3}
-\Delta U^m_{e,\ell}&=0&\quad&\text{in }\mathbb R^2\setminus\overline\omega\\
U^m_{e,\ell}&=Q_\ell^m&&\text{on }\gamma,\\
U^m_{e,\ell}(x)&=O (|x|^{-1})&&\text{as }|x|\to+\infty.
\end{alignat}
\end{subequations}
%Recall that $Q^m_\ell(x)=P^m_\ell(x)+c^m_\ell$, $P^m_1(x)=\Re((x_1+ix_2)^m)$, $P^m_2(x)=\Im((x_1+ix_2)^m)$ and $ c^m_\ell=-\int_\gamma  P_\ell^m(x)\psie{\gamma}(x)\,{\rm d}\sigma_x$. 
The functions $u^m_{e,\ell}:=U^m_{e,\ell}(\phi)$ are harmonic in $\mathbb R^2\setminus
\overline{D}$ ($D$ denotes the unit disk) and satisfy:
$$u^m_{e,1}(x)=\Re(\phi^m(z))+c^m_1\qquad\text{and}\qquad
u^m_{e,2}(x)=\Im(\phi^m(z))+c^m_2.$$
We can easily compute the constants $c^m_1$ and $c^m_2$ by writing that
$$c^m_1+ic^m_2=-\int_\gamma (P^m_1(x)+i\
P^m_2(x))\psie{\gamma}(x)\,{\rm d}\sigma_x=-\int_{-\pi}^{\pi}\phi^m(e^{it})|\phi'(e^{it})|\psie{\gamma}(e^{it})\,{\rm
	d}t.$$
But we know (from direct computations or from \cite[Theorem 17.3.3]{Hille:1962aa}) that $\psie{\gamma}(e^{it})=1/(2\pi|\phi'(e^{it})|)$. It follows
that $c^m_1+i c^m_2 =-a_0^m$ and 
we have, on the boundary of $D$:
$$
u^m_{e,1}(x)=\frac{1}{2}\left[\phi^m(z)+\overline{\phi^m(z)}\right]-\Re(a_0^m)\qquad\text{and}\qquad
u^m_{e,2}(x)=-\frac{i}{2}\left[\phi^m(z)-\overline{\phi^m(z)}\right]-\Im(a_0^m).$$
This can be rewritten, using the identity $\bar z=1/z$ on $\partial D$, as:
\begin{subequations}
	\begin{align}
	u^m_{e,1}(x)&=\frac{1}{2}\left[\overline{\phi^m_+}(z^{-1})+\phi^m_-(z)+\overline{\overline{\phi^m_+}(z^{-1})+\phi^m_-(z)}\right]\\
	u^m_{e,2}(x)&=-\frac{i}{2}\left[-\overline{\phi^m_+}(z^{-1})+\phi^m_-(z)-\left(-\overline{\overline{\phi^m_+}(z^{-1})+\phi^m_-(z)}\right)\right].
	\end{align}
\end{subequations}
These expressions lead us to introduce the following functions:
\begin{subequations}
	\label{def_lambda} 
	\begin{align}
	w^m_1(z)&=\overline{\phi^m_+}(z^{-1})+\phi^m_-(z)=\phi^m(z)-a_0^m+\lambda^m_1(z)\\
	w^m_2(z)&=-\overline{\phi^m_+}(z^{-1})+\phi^m_-(z)=\phi^m(z)-a_0^m+\lambda^m_2(z),
	\end{align}
\end{subequations}
where
\begin{equation}
	\label{def_lambda_2}
	\lambda^m_1(z)=\overline{\phi^m_+}(z^{-1})-\phi^m_+(z)\qquad\text{and}\qquad
	\lambda^m_2(z)=-\overline{\phi^m_+}(z^{-1})-\phi^m_+(z).
\end{equation}
The functions $w^m_1$ and $w^m_2$ are holomorphic in $\mathbb C\setminus\overline{D}$ and:
$$
u^m_{e,1}=\Re(w^m_1)\quad\text{and}\quad u^m_{e,2}=\Im(w^m_2)\quad\text{ on
}\partial D.$$
For every $X=(X_1,X_2)\in\mathbb R^2$ identified with $Z=X_1+iX_2\in\mathbb C$, we
have:
$$
\nabla U^m_{e,1}(\phi(z))\cdot X=\Re\left[(w^{m}_1)'(z)Z/\phi'(z)\right]\quad\text{ and
}\quad
\nabla U^m_{e,2}(\phi(z))\cdot X=\Im\left[(w^{m}_2)'(z)Z/\phi'(z)\right].
$$
On $\gamma$, the outer unit normal vector is parameterized by $t\in[-\pi,\pi[\longmapsto e^{it}{\phi'(e^{it})}/{|\phi'(e^{it})|}$,
and therefore, for every $m'\geqslant 1$:
\begin{equation}
\label{equ:part1}
\int_\gamma\partial_n U^m_{e,1}(x) \, Q^{m'}_{1}(x){\rm d}\sigma_{x}=\int_{-\pi}^{\pi}\Re\left[e^{it}(w^{m}_1)'(e^{it})\right]\, \Re\big[\phi^{m'}(e^{it})\big]{\rm d}t.
	\end{equation}
On the other hand, $U^m_{i,\ell}$ solves the following interior problem:
\begin{alignat*}{3}
-\Delta U^m_{i,\ell}&=0&\quad&\text{in }\omega\\
U^m_{i,\ell}&=Q_\ell^m&&\text{on }\gamma,
\end{alignat*}
whose unique solution is merely $U^m_{i,\ell}=Q_\ell^m$, so that:
\begin{equation}
\label{equ:part2}
\int_\gamma\partial_n U^{m}_{i,1}(x) Q^{m'}_{1}(x){\rm d}\sigma_{x}=\int_{-\pi}^{\pi}\Re\left[
e^{it}(\phi^{m})'(e^{it})\right]\Re\big[\phi^{m'}(e^{it})\big]{\rm d}t.
\end{equation}
Gathering now \eqref{equ:part1}, \eqref{equ:part2} and taking into account the expressions \eqref{def_lambda}, we infer that:
\begin{equation}
\label{eq:qmqmprime}
\langle  Q^{m}_1, Q^{m'}_1\ranglepp{\gamma}=
\int_\gamma [\partial_n U^m_{1}(x)]_{\gamma}\,Q^{m'}_{1}(x){\rm d}\sigma_{x}=
-\int_{-\pi}^{\pi}\Re\left[e^{it}(\lambda^{m}_1)'(e^{it})\right]\Re\big[\phi^{m'}(e^{it})\big]{\rm d}t.
\end{equation}
Notice now that, on $\partial D$, we have:
 $$\Re\left[e^{it}(\lambda^m_1)'(e^{it})\right]=-\left[\overline{e^{it}(\phi_+^m)'(e^{it})}+e^{it}(\phi_+^m)'(e^{it})\right]=-2\Re \left[e^{it}(\phi_+^m)'(e^{it})\right],$$ 
 and therefore, \eqref{eq:qmqmprime} can be rewritten as:
$$\langle  Q^{m}_1, Q^{m'}_1\ranglepp{\gamma}=2
\int_{-\pi}^{\pi}\Re \left[e^{it}(\phi_+^m)'(e^{it})\right]\Re\left[\phi^{m'}(e^{it})\right]{\rm
	d}t.$$
Using similar arguments, lengthy but straightforward computations lead to:
\begin{align*}
\langle  Q^{m}_1, Q^{m'}_2\ranglepp{\gamma}&=2
\int_{-\pi}^{\pi}\Re \left[e^{it}(\phi_+^m)'(e^{it})\right]\Im\left[\phi^{m'}(e^{it})\right]{\rm
	d}t\\
	&=2
\int_{-\pi}^{\pi}\Im \left[e^{it}(\phi_+^m)'(e^{it})\right]\Re\left[\phi^{m'}(e^{it})\right]{\rm
	d}t\intertext{and}
	\langle  Q^{m}_2, Q^{m'}_2\ranglepp{\gamma}&=2
	\int_{-\pi}^{\pi}\Im \left[e^{it}(\phi_+^m)'(e^{it})\right]\Im\left[\phi^{m'}(e^{it})\right]{\rm
	d}t.
\end{align*}
Formulae \eqref{main_mu_nu} follow.
			\end{proof}

			%----------------------------------------------------------------------
			\begin{lemma}
			The following relations hold true:
			$$a_1=\sqrt{\frac{\mu_1}{2\pi}}\quad\text{and}\quad a_0=\frac{\mu_2}{2\mu_1}.$$
			\end{lemma}		
			%----------------------------------------------------------------------
			\begin{proof}
			For every $z\in\mathbb C\setminus \overline{D}$ (recall that $D$ denotes the unit disk), we have $\phi^1_+(z)=\phi_+(z)=a_1 z$ and $\phi^2_+(z)=(a_1)^2 z^2+2a_1a_0 z$, and hence
			$$z\phi_+'(z)=a_1 z\qquad\text{and}\qquad z(\phi^2_+)'(z)=2(a_1)^2z^2+2a_1 a_0.$$
			Applying formulae \eqref{main_mu_nu}, we obtain:
			$$\mu_1=2\pi (a_1)^2\qquad\text{and}\qquad \mu_2=4\pi (a_1)^2 a_0.$$
			The conclusion of the lemma follows.
			\end{proof}
			%----------------------------------------------------------------------
			The conformal mapping $\phi^{-1}$ can be expanded as a Laurent's series having the form:
			$$\phi^{-1}(z)=b_1 z+b_0+\sum_{k\leqslant -1}b_k z^k,$$
			outside a disk $D'$ centered at the origin and containing $\omega$.  The complex coefficients $b_k$ ($k\leqslant 1$) can be deduced on the one hand 
			from the coefficients $a_k$ of $\phi$, and on the other hand from the values of $\nu_{m}$, ($m\geqslant 1$), as claimed in the following lemma:
			%----------------------------------------------------------------------
			\begin{lemma}
			The following relations hold true:
			\begin{equation}
\label{def_b_1}
b_1=1/a_1\quad\text{and}\quad b_0=-a_0/a_1.
\end{equation}
For every $m\geqslant 1$, we have:
\begin{equation}
\label{estim_b}
b_{-m}=-\frac{\nu_{m}}{2\pi
	a_1m}=-\frac{1}{m}\sum_{|\beta|=-1}a_{\beta_1}\ldots
a_{\beta_m},\qquad m\geqslant 1.
\end{equation}
			\end{lemma}
			%----------------------------------------------------------------------
			\begin{proof} Identities \eqref{def_b_1} follow straightforwardly because $\phi$ and $\phi^{-1}$ are inverse mappings. Integrating by part the expression of $\nu_{m}=\nu^{m,1}$ obtained from \eqref{main_mu_nu}, we get:
\begin{equation}
\label{nu_int0}
\nu_{m}=a_1\int_{-\pi}^{\pi}e^{it}\phi^{m}(e^{it})\,{\rm
	d}t=-a_1m\int_{-\pi}^{\pi}e^{it}\phi'(e^{it})e^{it}\phi^{m-1}(e^{it})\,{\rm
	d}t.
\end{equation}
Since $t\mapsto e^{it}$ is a parameterization of $\partial D$, applying  Cauchy's integral formula we get on the one hand
$$
\int_{-\pi}^{\pi}e^{it}\phi^{m}(e^{it})\,{\rm	d}t=-i \int_{\partial D}\phi^m(\xi){\rm d}\xi =2\pi a_{-1}^m.
$$
On the other hand, since $t\mapsto\phi(e^{it})$ is a parameterization of $\gamma$ and the function $\phi^{-1}$ being holomorphic in $\mathbb C\setminus \overline{\omega}$, we have:
$$
\int_{-\pi}^{\pi}e^{it}\phi'(e^{it})e^{it}\phi^{m-1}(e^{it})\,{\rm	d}t =-i \int_{\gamma}\phi^{-1}(\xi)\xi^{m-1}\,{\rm d}\xi=-i \int_{\partial D'}\phi^{-1}(\xi)\xi^{m-1}\,{\rm d}\xi =2\pi b_{-m}.
$$
Identity \eqref{nu_int0} can thus be rewritten as:
$$\nu_{m}=2\pi a_1 a_{-1}^m=-2\pi a_1 mb_{-m},$$
and identity \eqref{estim_b} follows according to \eqref{def_an}. 
\end{proof}

Using the above lemmas, 	we are in position to prove the main result of this section, namely Theorem~\ref{thm:expl}.	
		
		%----------------------------------------------------------------------
\begin{proof}[of Theorem~\ref{thm:expl}]
Since $\phi$ and $\phi^{-1}$ play symmetric roles, we can exchange $a_m$ and $b_m$ in Formula \eqref{estim_b} to obtain:
$$
a_{-m}=-\frac{1}{m}\sum_{|\beta|=-1}b_{\beta_1}\ldots
b_{\beta_m},\qquad m\geqslant 1.
$$
Reordering the terms of the above sum, we get that
\begin{equation}
\label{estim_a}
a_{-m} = -\frac{1}{m}\sum_{(\theta,\alpha)\in \mathsf B_{m}} b_{1}^{\theta}b_{0}^{\alpha_0}b_{-1}^{\alpha_{1}}\dots b_{-m}^{\alpha_{m}}
\end{equation}
where $\mathsf B_{m}$ is the set of  $(\theta,\alpha)\in {\mathbb N}\times {\mathbb N}^{m+1}$ such that
\begin{align*}
\theta + \alpha_{0} + \alpha_{1} +\dots+\alpha_{m} &=m\\
\theta -  ( \alpha_{1} +2\alpha_{2}+\dots+m\alpha_{m}) &=-1.
\end{align*}
Now, one can easily check that $(\theta,\alpha)\in \mathsf B_{m}$ if and only if $\alpha$ belongs to the set $\mathsf A_{m}$ defined by \eqref{eq:An} and $\theta=m-(\alpha_{0} + \alpha_{1} +\dots+\alpha_{m})$. Therefore, \eqref{estim_a} also reads
$$
a_{-m} = -\frac{1}{m}\sum_{\alpha\in \mathsf A_{m}} b_{1}^{m-(\alpha_{0} + \alpha_{1} +\dots+\alpha_{m})}b_{0}^{\alpha_0}b_{-1}^{\alpha_{1}}\dots b_{-m}^{\alpha_{m}}
$$
Using \eqref{def_b_1} and the first equality of \eqref{estim_b} in the above relation, we obtain that
$$
a_{-m} =- \frac{1}{m} a_{1}^{-m} \sum_{\alpha\in \mathsf A_{m}} (-a_{0})^{\alpha_0}
\left(-\frac{\nu_{1}}{2\pi}\right)^{\alpha_1}\dots \left(-\frac{\nu_{m}}{2\pi m}\right)^{\alpha_m},
$$
and the conclusion follows immediately.
\end{proof}

%==============================================================================
\subsection{About stability}
It is well-known that logarithmic stability is best possible for Calder\'on's inverse problem. In the particular case of cavities, this result is proved in \cite[Theorem 4.1]{AleRon01} where the error on the geometry (measured using the Hausdorff distance) is estimated in terms of the error of the DtN (measured in operator norm). 
However, as suggested by Alessandrini and Vessella \cite{AleVes05}, one can try to construct stable functionals, namely Lipschitz-continuous functions of the data carrying relevant information on the geometry of the obstacle. According to formula \eqref{la_haut}, each coefficient $a_{k}$, $k\leqslant 1$, yields an example of such functional. Actually, we can prove that each coefficient is not only a Lipschitz-continuous function of the data, but is analytic. Let us define the following open subspace of $\mathcal L(H(\Gamma))$:
$$\mathcal U_\Gamma = \{\mathsf R\in  \mathcal L(H(\Gamma))\, : \,  {\rm Id}+\mathsf R \text{ invertible and } 
\mu_{1}(\mathsf R) >0\},$$
where $\mu_{1}(\mathsf R) :=\langle Q^1_{\Gamma},({\rm Id}+\mathsf R)^{-1}\mathsf R\overline {Q^{1}_{\Gamma}}\ranglepp{\Gamma}$.

Notice in particular that, for every Lipschitz Jordan curve $\gamma$, the continuous linear mapping $\mathsf R = \mathsf S_\Gamma(\Lambda_\gamma-\Lambda_0)$ belongs to $\mathcal U_\Gamma$.  
We deduce straightforwardly  the following analyticity result:
%---------------------------------------------------------------------------------------------------------------------
\begin{theorem}
On the open subset $\mathcal U_\Gamma$ of $\mathcal L(H(\Gamma))$ define the sequence of analytic functions $a_k:\mathcal U_\Gamma\to \mathbb C$ ($k\leqslant 1$) as given by the formulae \eqref{la_haut}. If 
$\mathsf R=\mathsf S_\Gamma(\Lambda_\gamma-\Lambda_0)$ for some Lipschitz Jordan curve $\gamma$, a parameterization of $\gamma$ is given by:
$$t\in]-\pi,\pi]\mapsto \sum_{k\leqslant 1} a_k(\mathsf R) e^{ikt}\in\mathbb C.$$
\end{theorem}

%==============================================================================
\section{Numerical results}
\label{sect:NUM}
%==============================================================================
We present in this section some numerical experiments meant to illustrate the feasibility of the proposed  reconstruction method. For the sake of clarity, let us first sum up the different steps of the simple reconstruction algorithm:
\begin{enumerate}
\item Compute a numerical approximation of the operator $\mathsf R=\mathsf S_\Gamma(\Lambda_\gamma-\Lambda_0)$.
\item Fix an integer $M$ and compute for $1\leqslant m\leqslant M$ 
$$
\mu_m:=\frac{1}{2}\langle Q^m_\gamma,\overline {Q^1_\gamma}\rangle_{\frac12,\gamma} =\frac12 \langle Q^m_\Gamma,({\rm Id}+\mathsf R)^{-1}\mathsf R \overline {Q^1_\Gamma} \rangle_{\frac 12, \Gamma}
$$
$$
\nu_m:=\frac{1}{2}\langle Q^m_\gamma, Q^1_\gamma\rangle_{\frac12,\gamma}=\frac{1}{2}\langle Q^m_\Gamma,({\rm Id}+\mathsf R)^{-1}\mathsf R Q^1_\Gamma \rangle_{\frac 12, \Gamma}.
$$
\item Compute $(a_{-m})_{-1\leqslant m\leqslant M}$ via formulae \eqref{la_haut}.
\item Plot the image of the unit circle by $$\phi_M(z)= a_1 z+a_0+\sum_{1\leqslant m\leqslant M} a_{-m} z^{-m}.$$
\end{enumerate}
Let us give some details about the implementation. We use the finite dimensional approximation space spanned by the family  $\mathcal Q_{\Gamma}^M:=\{Q^m_{\Gamma}, \, \overline {Q^m_{\Gamma}},\, 1\leqslant m \leqslant M\}$. We denote by  $\mathbf Q_{\Gamma}$ the $2M\times 2M$ complex matrix whose entries are the $\langle f,g\ranglepp{\Gamma}$, where $(f,g)\in \mathcal Q_{\Gamma}^M\times \mathcal Q_{\Gamma}^M$. Note that $\mathbf Q_{\Gamma}$ is nothing but the Generalized Polya-Szeg\"o Tensor (GPST) associated to $\Gamma$. Obviously, a similar matrix  $\mathbf Q_{\gamma}$ can be defined for the boundary $\gamma$. We denote by $ \mathbf R$ the matrix whose entries are $\langle f,\mathsf R g\ranglepp{\Gamma} = \langle (\Lambda_\gamma-\Lambda_0) g,f\ranglemp{\Gamma}$, for  $(f,g)\in \mathcal Q_{\Gamma}^M\times \mathcal Q_{\Gamma}^M$. With this notation, the reader can easily check that formula \eqref{factor:main_varia} admits the following discrete version:
\begin{equation}
\label{belle_formule}
\mathbf Q_{\gamma} \simeq \widetilde{\mathbf Q}_{\gamma}:=  \mathbf Q_{\Gamma}  (\mathbf Q_{\Gamma} + \mathbf R)^{-1} \mathbf R.
\end{equation}
This formula relates in a very simple way, through the measurement operator $\mathbf R$, the GPST of $\gamma$ to the GPST of $\Gamma$. In particular, the coefficients $\mu_m$ and $\nu_m$ are particular entries of $\frac12 \mathbf Q_{\gamma} $.

We consider now a test configuration in which $\Gamma$ is an ellipse centered at the origin and of major axis $[-1.9,1.9]$ and minor axis $[-1.1,1.1]$. The boundary $\gamma$ of the obstacle is parameterized by: 
$$t\in]-\pi,\pi]\mapsto \sum_{k=-7}^1 a_k e^{ikt},$$
where the complex coefficients $a_k$ are given in the following table:
$$
\begin{array}{|c|c|c|c|c|c|c|c|c|}
\hline
a_1&a_0&a_{-1}&a_{-2}&a_{-3}&a_{-4}&a_{-5}&a_{-6}&a_{-7}\\ \hline
 0.5&-1&0.085& - 0.06i  &-0.035&   0.06i&0& - 0.01i&-0.005\\ \hline
\end{array}
$$

The data are generated using the \verb?Matlab? Laplace boundary integral equation solver (for more information, see this link: \href{http://www.iecn.u-nancy.fr/~munnier/IES/}{IES}). Taking $M=12$, we first show on Figure~\ref{fig:test1} the reconstructed cavity for exact data and using the eight coefficients $a_1,\dots,a_{-6}$.
\begin{figure}[h]
	\centerline{\includegraphics[height=0.39\textwidth]{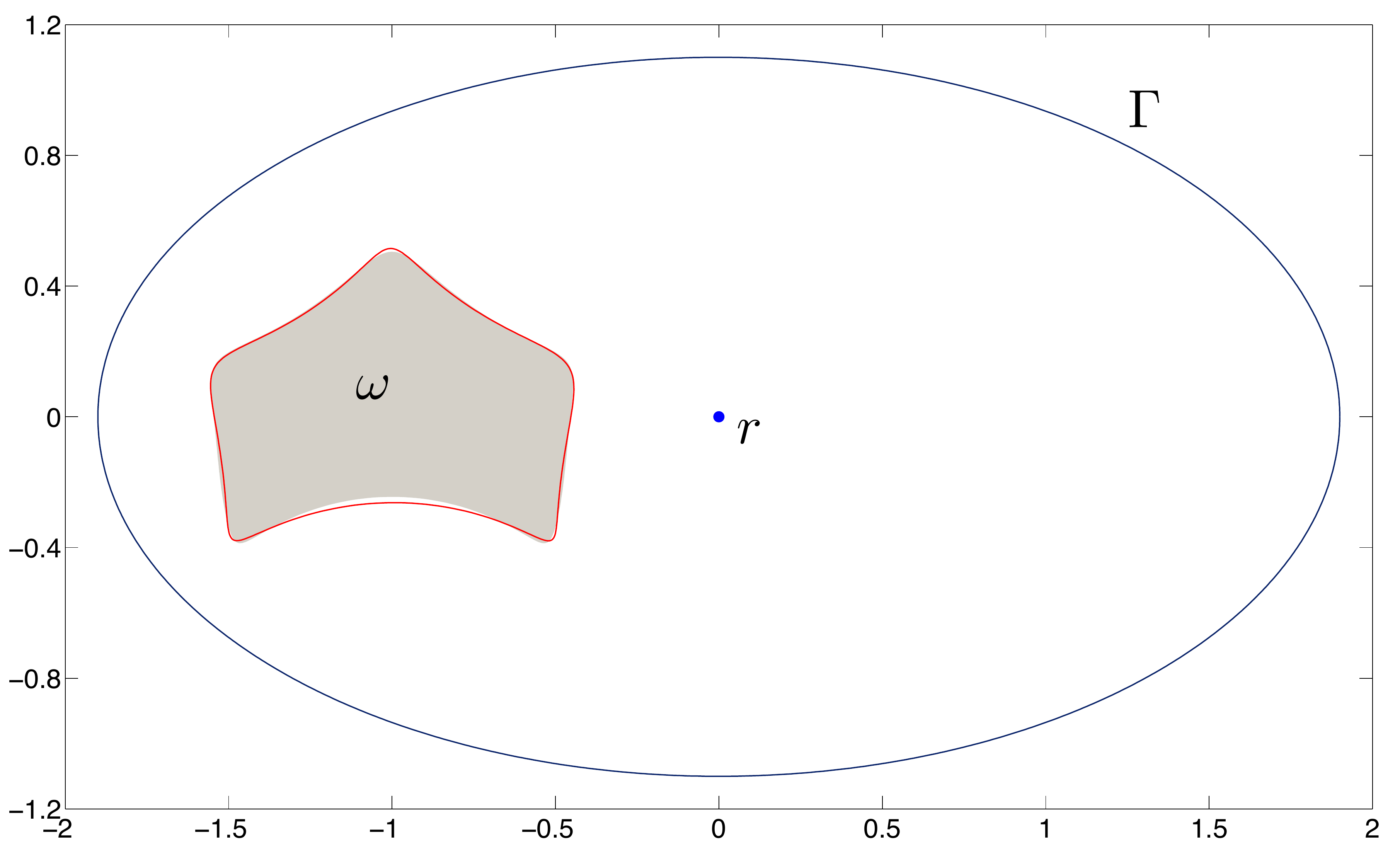}}
	\caption{Typical configuration: reconstruction with $a_1,\ldots,a_{-6}$ (in red) and actual inclusion (in gray). The blue point stands for the position of the origin $r$.}
	\label{fig:test1}
\end{figure}

Instead of using the harmonic polynomials $z^{n}$ in Definition~\ref{def:PQ}, one can use the shifted harmonic polynomials $(z-r)^{n}$, for some given $r\in \mathbb C$. This additional parameter turns out to have some influence on the quality of the reconstructed cavities, as shown in Figure~\ref{fig:r}. For instance, choosing $r$ in the neighborhood of $-0.5$, one can recover the six coefficients $a_{1}, \dots, a_{-4}$ with a relative error less than 2\%, while for $r=0$ this accuracy is achieved only for the coefficients $a_{1}, \dots, a_{-2}$. 

\begin{figure}[h]
	\centerline{\includegraphics[height=0.39\textwidth]{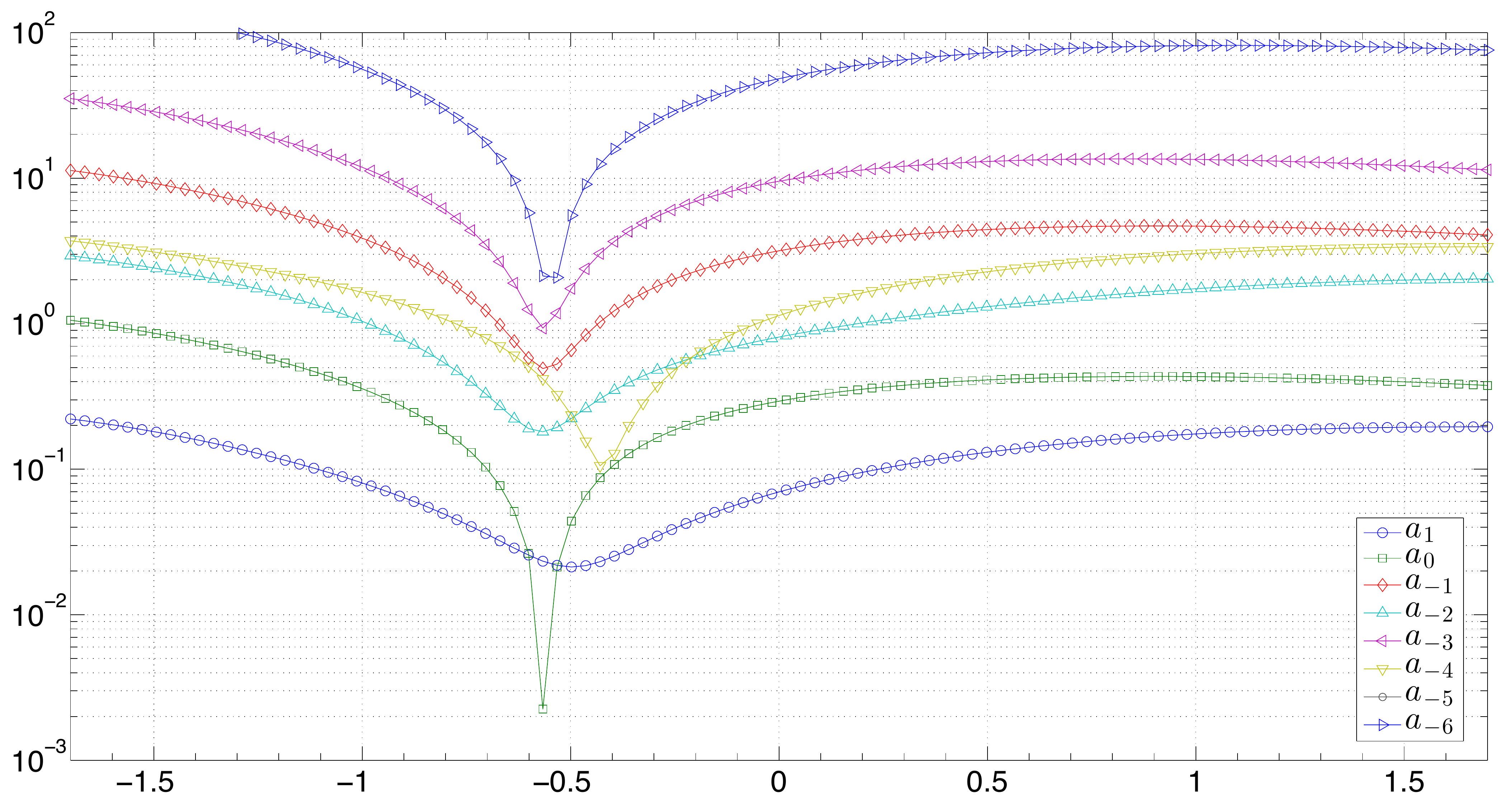}}
	\caption{Relative error of the retrieved coefficients (in $\%$) with respect to the abscissa of $r$ (there is no relative error for the coefficient $a_{-5}$ because it is null).}
	\label{fig:r}
\end{figure}

Let us take now $r=-0.5$ and consider a more realistic configuration of noisy data. We generate a random matrix $\mathbf N$ having the same size as $\mathbf R$ and whose coefficients are uniformly distributed between $-1$ and $1$. For $\delta = 0.05, 0.15, 0.25,0.35$, we compute the matrix $\mathbf R^N$ whose coefficients are:
$$R^N_{ij}=(1+\delta N_{ij})R_{ij},\qquad 1\leqslant i,j\leqslant 2M,$$
and we replace $\mathbf R$ by $\mathbf R^N$ in formula \eqref{belle_formule}.

We show on figures~\ref{fig:noise}-\ref{fig:noise3} examples of reconstructed cavities respectively with $5\%, 15\%, 25\%$ and $35\%$ of noise. The number of correctly recovered coefficients decreases with the level of noise and only those coefficients are used in the reconstruction. We plot on Figure~\ref{fig:noisy_data} the dependence of the  mean relative error with respect to the level of noise and we notice a good stability of the first three coefficients $a_{1}, a_{0}$ and $a_{-1}$. 
\begin{figure}[h]
	\centerline{\includegraphics[height=0.39\textwidth]{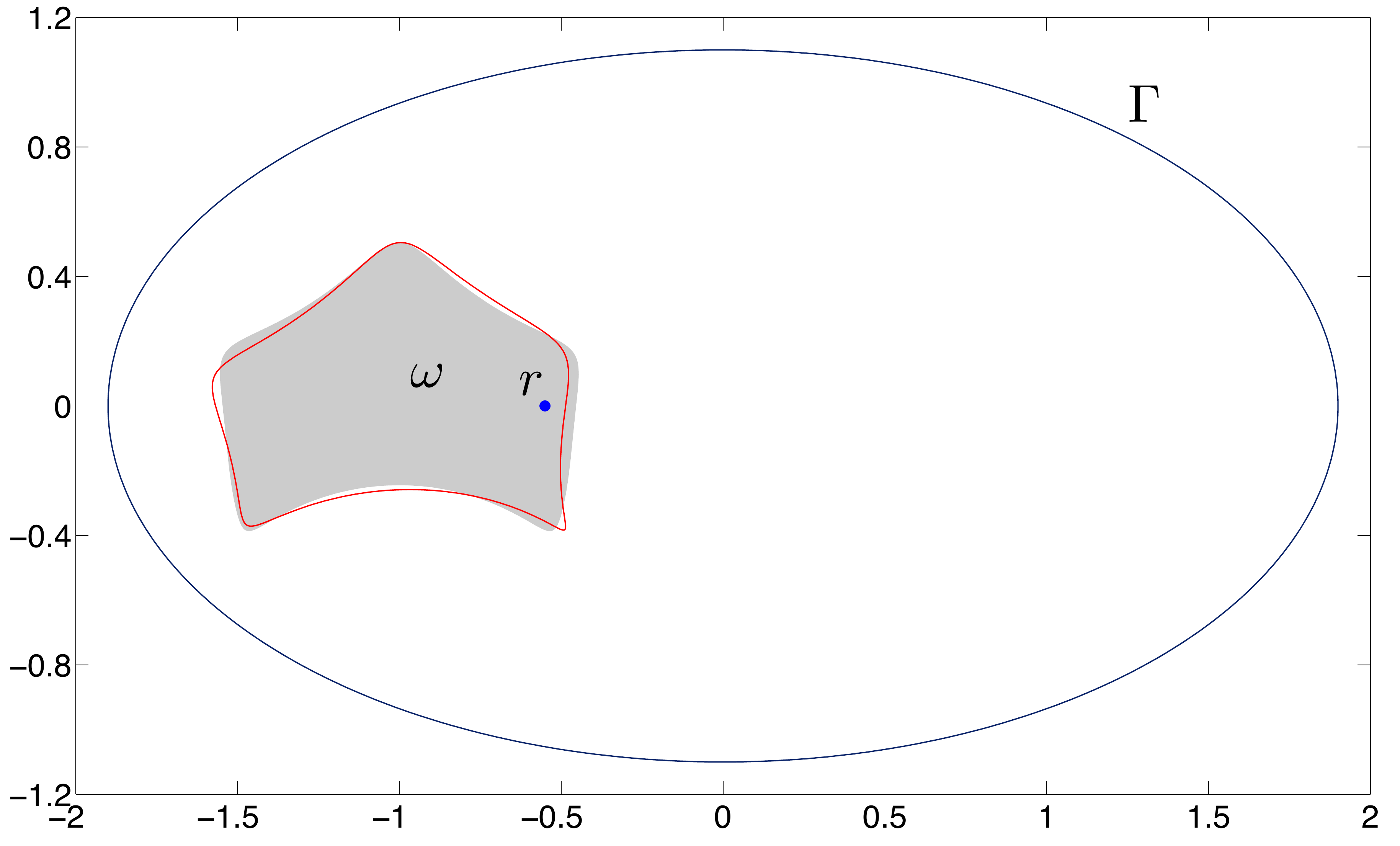}}
	\caption{Reconstruction (in red) with $a_1,\ldots,a_{-4}$ and actual inclusion (in gray) with $5\%$ noise.}
	\label{fig:noise}
\end{figure}
\begin{figure}[h]
	\centerline{\includegraphics[height=0.39\textwidth]{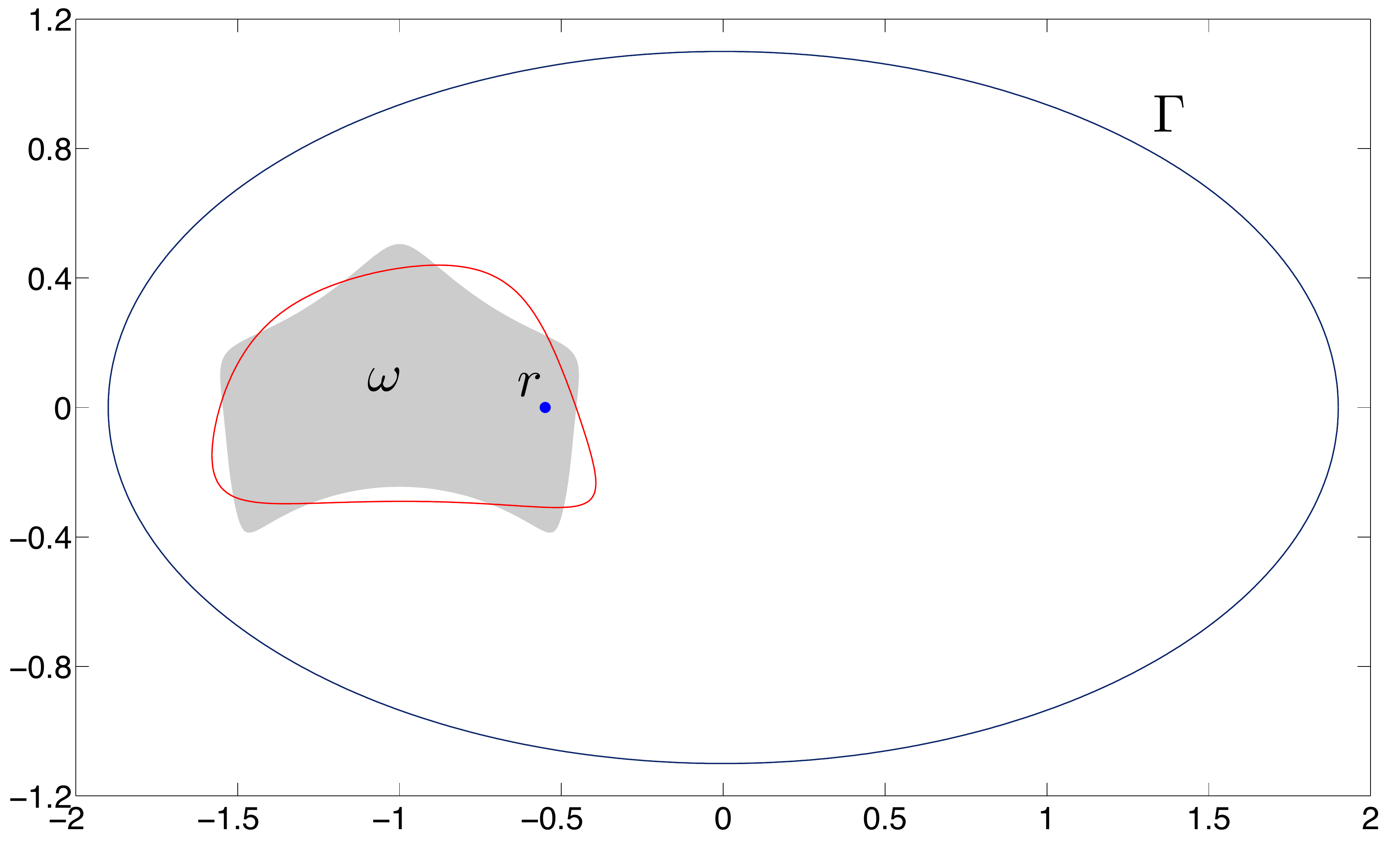}}
	\caption{Reconstruction (in red) with $a_1,\ldots,a_{-4}$ and actual inclusion (in gray) with $15\%$ noise.}
	\label{fig:noise1}
\end{figure}

\begin{figure}[h]
	\centerline{\includegraphics[height=0.39\textwidth]{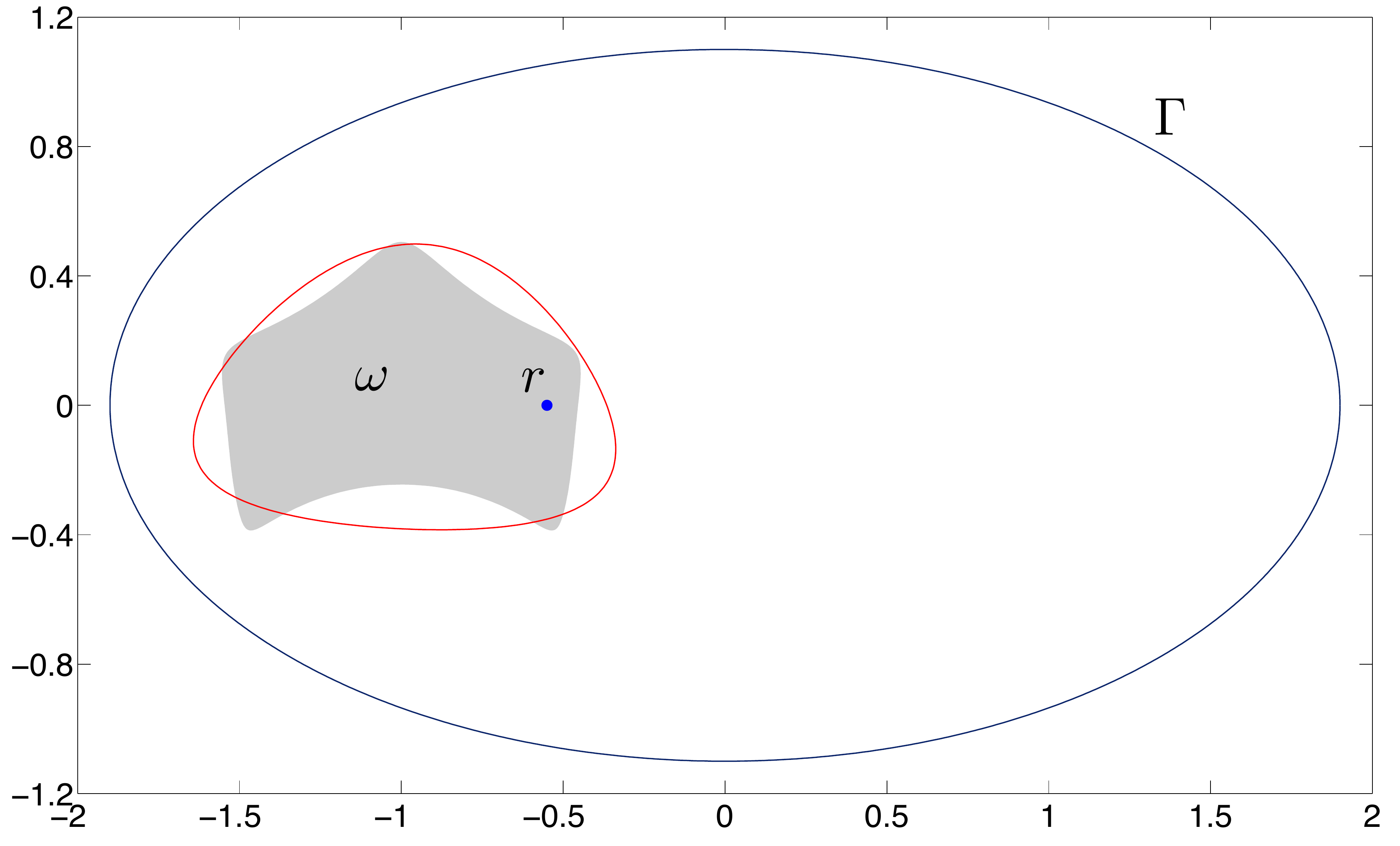}}
	\caption{Reconstruction (in red) with $a_1,\ldots,a_{-2}$ and actual inclusion (in gray) with $25\%$ noise.}
	\label{fig:noise2}	
\end{figure}
\begin{figure}[h]
	\centerline{\includegraphics[height=0.39\textwidth]{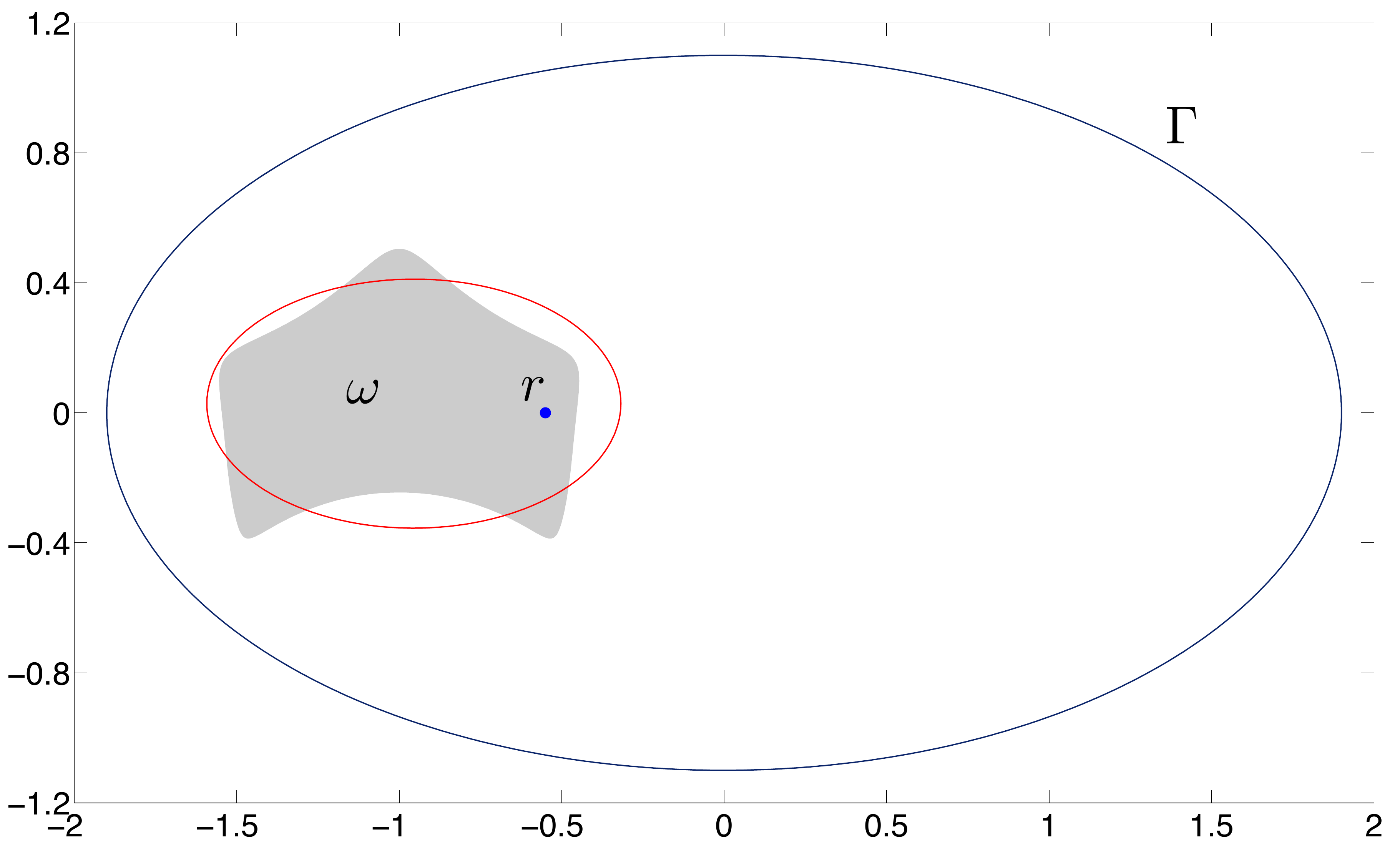}}
	\caption{Reconstruction (in red) with $a_1,\ldots,a_{-1}$ and actual inclusion (in gray) with $35\%$ noise.}
	\label{fig:noise3}
\end{figure}

%We have proved that every coefficient $a_k$ ($k\leqslant 1$) is stable. We consider now noisy data:
%$$\mathbb S_{ij} = \mathbb S_{ij} + \mathbb N_{ij}\mathbb S_{ij}$$

\begin{figure}[h]
	\centerline{\includegraphics[height=0.39\textwidth]{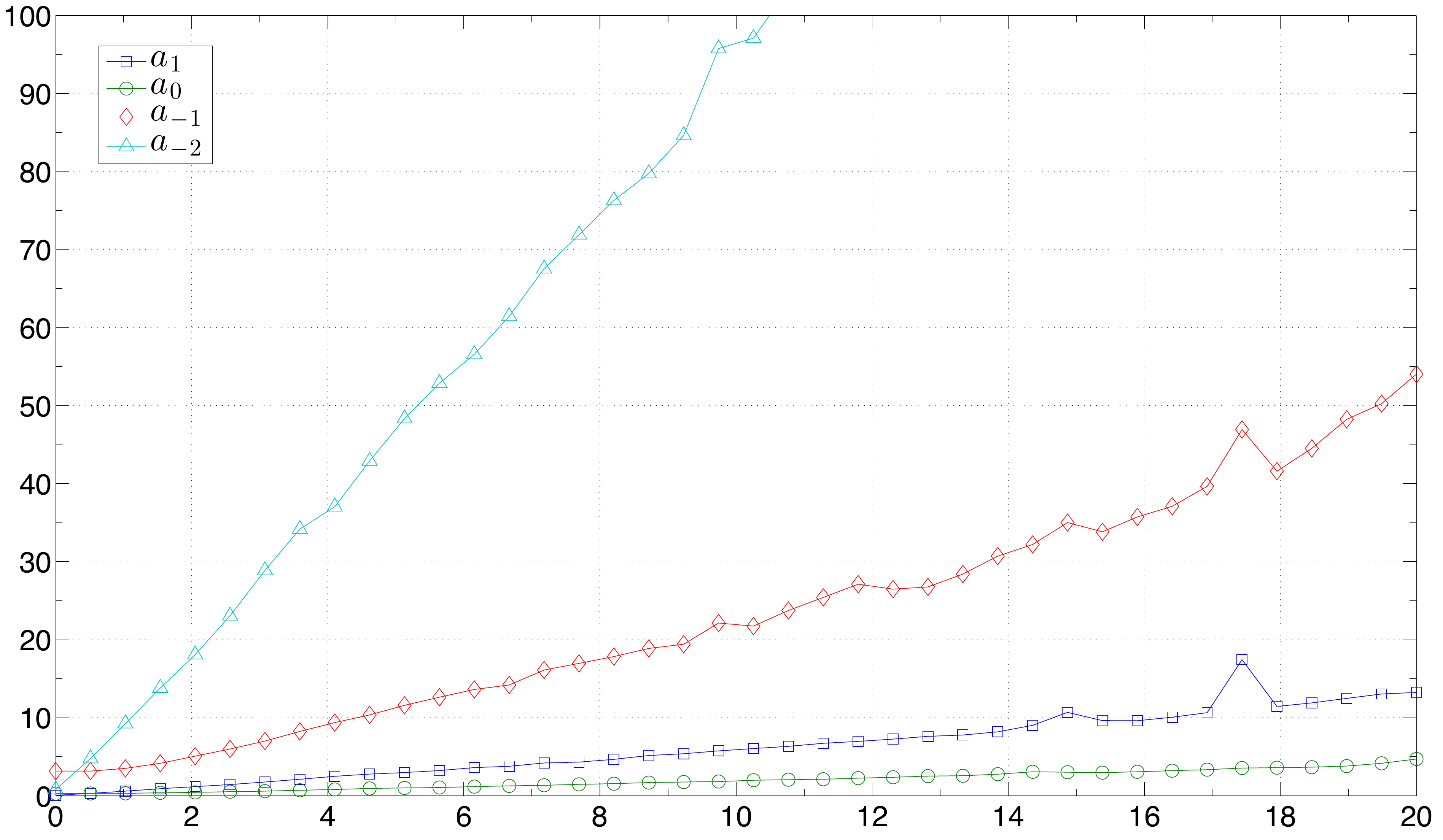}}
	\caption{Relative error of the retrieved coefficients (in $\%$)  with respect to the level of noise (in $\%$) of the data. The point $r$ is the center of the ellipse. Notice the stability of the conformal center $a_0$ and the logarithmic capacity $a_1$.}
	\label{fig:noisy_data}	
\end{figure}

Finally, we illustrate on Figure~\ref{fig:other} the efficiency of the method for more complex geometries (non convex outer boundary $\Gamma$ and a non centered cavity). The choice of the parameter $r$ to obtain good reconstructions is not clear so far and this would need to be further investigated.

\begin{figure}[h]
	\centerline{\includegraphics[height=0.39\textwidth]{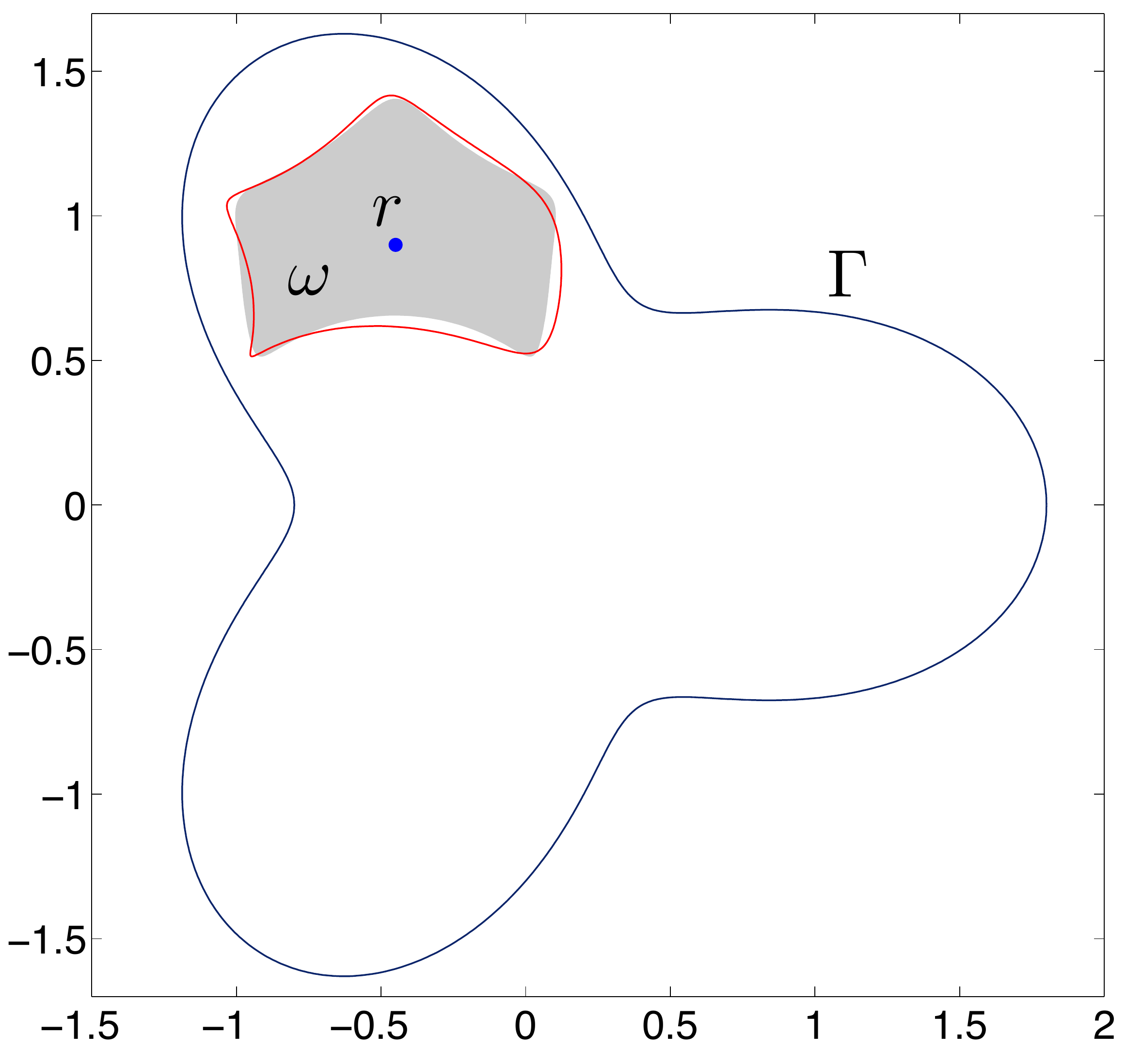}\qquad\includegraphics[height=0.39\textwidth]{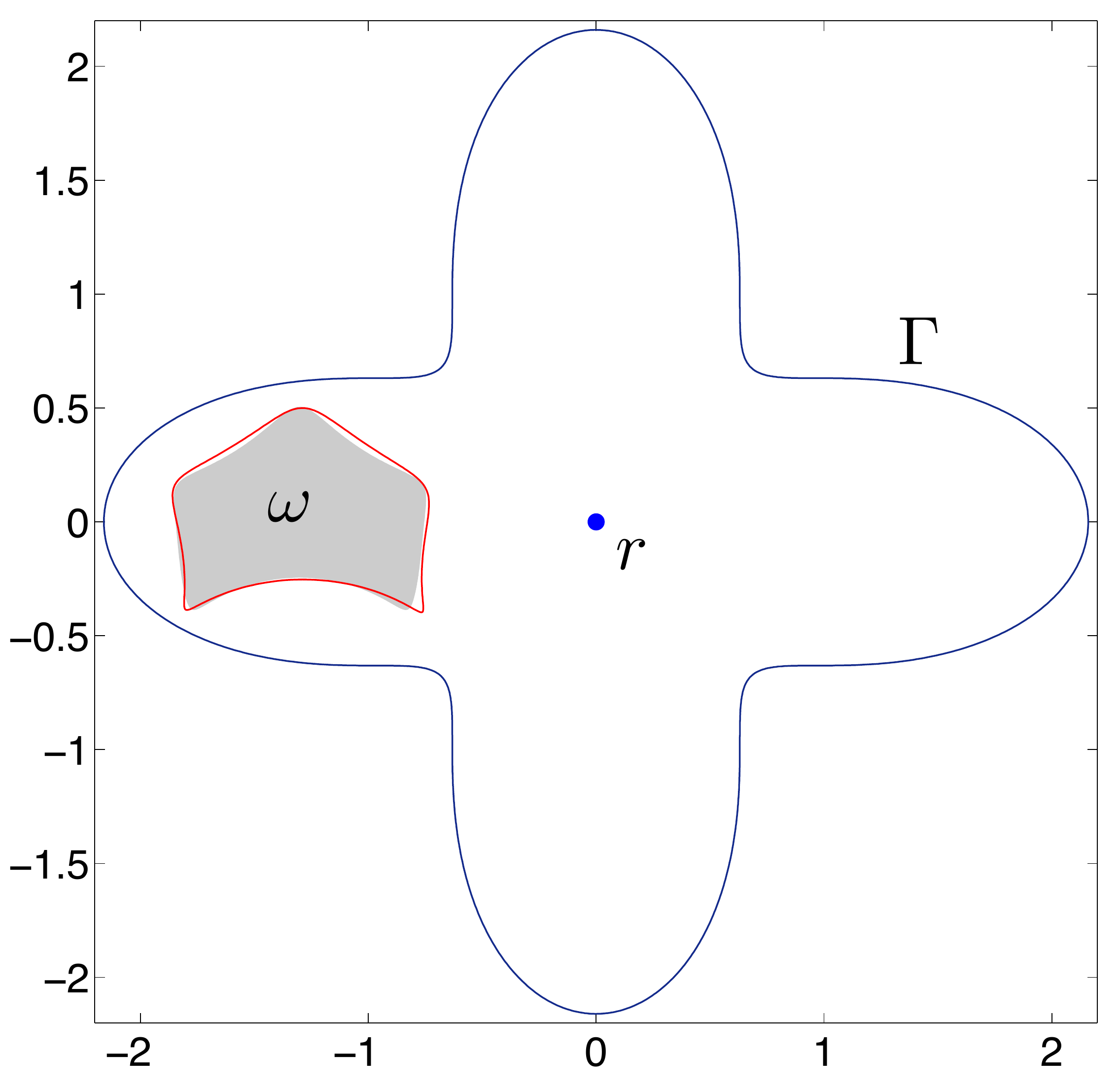}}
	\caption{Other examples of reconstruction with more complex boundary $\Gamma$.}
	\label{fig:other}		
\end{figure}

\appendix
%==============================================================================
\section{Appendix}
%==============================================================================

Consider problem \eqref{calderon} with a piecewise conductivity $\sigma(x) = 1+ (\alpha-1) 1_{\omega}(x)$ (where $1_{\omega}$ denotes the characteristic function of $\omega$ and $\alpha$ a positive constant).

%We set  $\mathsf C_i=\Omega\setminus\overline{\omega}$, $\mathsf C_e=\mathbb
%R^2\setminus\overline{\mathsf C}_i$ and $\mathscr C=\partial\mathsf
%C=\Gamma\cup\gamma$. Notice that the $\mathsf C_e$ consists in two connected
%components: an unbounded one, denoted by $\mathsf C_e^+$ and a bounded one, 
%denoted by $\mathsf C_e^-$.

\begin{prop}
\label{propppp}
	For every $f\in H^{\frac{1}{2}}(\Gamma)$, System \eqref{main_problem}  admits a
	unique solution $(u^f,c^f)\in H^1(\Omega\setminus\overline{\omega})\times\mathbb R$. It is the unique pair that realizes:
	\begin{equation}
	\label{min_1}
	\min_{(u,c)\in H^1(\Omega\setminus\overline{\omega})\times\mathbb R}\left\{\frac{1}{2}\int_{\Omega\setminus\overline{\omega}}|\nabla
	u|^2{\rm d}x\,:\,u|_{\Gamma}=f,\,u|_\gamma=c\right\}.
	\end{equation}
	The function $u^f$ can be considered as a function of $H^1(\Omega)$ by setting $u^f=c^f$ in $\omega$.
	
	For every $\alpha>0$, System \eqref{calderon}  admits a unique solution $u^f_\alpha\in H^1(\Omega)$. This function achieves:
	\begin{equation}
	\label{min_1_b}
	\min_{u\in H^1(\Omega)}\left\{\frac{1}{2}\int_{\Omega\setminus\overline{\omega}}|\nabla
	u|^2{\rm d}x+\frac{\alpha}{2}\int_{\omega}|\nabla
	u|^2{\rm d}x\,:\,u|_{\Gamma}=f\right\}.
	\end{equation}
	The following convergence result holds true for every $f\in\Hp{\Gamma}$:
	$$u_\alpha^f\to u^f\qquad\text{in }H^1(\Omega)\text{ as }\alpha\to+\infty.$$
	\end{prop}
	%-------------------------------------------------------------------------------------------------------------------
	\begin{proof}
The minimization problem \eqref{min_1} can be reformulated as:
	\begin{equation}
	\label{min_2}
	\min_{(w,c)\in H^1_0(\Omega\setminus\overline{\omega})\times\mathbb R}\int_{\Omega\setminus\overline{\omega}}|\nabla (w +
	e^f+ cv)|^2{\rm d}x,
	\end{equation}
	where $e^f$ and $v$ are both harmonic in $\Omega\setminus\overline{\omega}$ with Dirichlet data
	$e^f|_{\Gamma}=f$, $e^f|_\gamma=0$ and $v|_\Gamma=0$, $v|_\gamma=1$. For every 
	$w\in H^1_0(\Omega\setminus\overline{\omega})$, we get:
	$$\int_{\Omega\setminus\overline{\omega}}|\nabla (w + e^f+ cv)|^2{\rm d}x=\int_{\Omega\setminus\overline{\omega}}|\nabla
	w|^2{\rm d}x+\int_{\Omega\setminus\overline{\omega}}|\nabla e^f|^2{\rm d}x+c^2
	\int_{\Omega\setminus\overline{\omega}}|\nabla v|^2{\rm d}x+2c \int_{\gamma}\partial_n e^f{\rm
		d}\sigma,$$
	and therefore the minimum in \eqref{min_2} is unique and achieved for $w=0$ and 
	$$c^f = - \int_{\gamma}\partial_n e^f{\rm d}\sigma\left(\int_{\Omega\setminus\overline{\omega}}|\nabla v|^2{\rm d}x\right)^{-1}.$$
	The corresponding  fonction $u^f:= e^f+c^f v$ is the unique minimizer of problem
	\eqref{min_1}, and can easily be shown to solve System \eqref{main_problem}. It
	is classical to verify that, reciprocally, every solution of System \eqref{main_problem} provides a solution
	to the minimization problem \eqref{min_1}.
	
	Seeking the minimum of problem \eqref{min_1_b} in the form $u=u^f+w$ with $w\in H^1_0(\Omega)$, we are led to consider the new, equivalent, minimization problem:
	\begin{equation}
	\label{min_2_b}
	\min_{w\in H^1_0(\Omega)}\int_{\Omega\setminus\overline{\omega}}|\nabla w|^2{\rm d}x + \frac{\alpha}{2}\int_{\omega}|\nabla w|^2{\rm d}x+\int_\gamma \partial_n u^f w\,{\rm d}\sigma,
	\end{equation}
	where we have used the fact that $u^f=c^f$ in $\omega$. The existence and uniqueness for such a problem is straightforward and we denote by $w_\alpha^f$ the minimizer. 
	Introducing 
	$${\overline w}^f_\alpha=\frac{1}{{\rm mes}(\omega)}\int_\omega w^f_\alpha\,{\rm d}x,$$
	and taking into account the condition \eqref{free_circ}, the last term in the right hand side can be rewritten as:
	$$\int_\gamma \partial_n u^f w^f_\alpha\,{\rm d}\sigma = \int_\gamma \partial_n u^f (w^f_\alpha-{\overline w}^f_\alpha)\,{\rm d}\sigma.$$
	Invoking the Poincar\'e-Wirtinger inequality in $\omega$, we get the estimate:
	\begin{equation}
	\label{poinca}
	\left|\int_\gamma \partial_n u^f w^f_\alpha\,{\rm d}\sigma\right|\leqslant C\|\nabla w^f_\alpha\|_{L^2(\omega)},
	\end{equation}
	where the constant $C>0$ depends only upon $\omega$. The minimum \eqref{min_2_b} is negative ($w=0$ is an admissible function), whence we deduce that:
	$$\frac{\alpha}{2}\|\nabla w^f_\alpha\|_{L^2(\omega)}\leqslant C,$$
	and therefore 
	\begin{equation}
	\label{wfto0}
	\|\nabla w^f_\alpha\|_{L^2(\omega)}\to 0\quad\text{as}\quad \alpha\to+\infty.
	\end{equation}
	Remarking again that the minimum \eqref{min_2_b} is negative and using the estimate \eqref{poinca} together with the convergence result \eqref{wfto0}, we deduce that:
	$$\|\nabla w^f_\alpha\|_{L^2(\Omega\setminus\overline{\omega})}\to 0\quad\text{as}\quad\alpha\to+\infty,$$
	and the proof is completed.
\end{proof}

\end{document}